\documentclass[a4paper, reqno, 12pt]{amsart}
\usepackage[utf8]{inputenc}
\usepackage[T1, T2A]{fontenc}
\usepackage[english]{babel}
\usepackage{amsmath, amssymb, amsfonts, amsthm, amscd, mathrsfs}
\usepackage{enumitem}
\usepackage[hidelinks,hypertexnames=false]{hyperref}
\usepackage{longtable}
\setlist[enumerate]{label = (\alph*), ref=(\text{\alph*)}}
\setlist[itemize]{nolistsep,leftmargin=1.5\parindent}
\usepackage{ifthen}
\usepackage{multicol}
\usepackage{paracol}
\usepackage{comment}

\usepackage{epsfig}
\usepackage{graphicx}
\usepackage[cmtip,arrow]{xy}
\usepackage{bm}

\renewcommand{\phi}{\varphi}
\renewcommand{\ge}{\geqslant}
\renewcommand{\le}{\leqslant}

\newcommand{\KK}{\mathbb{K}}

\newcommand{\ZZ}{\mathbb{Z}}

\newcommand{\GG}{\mathbb{G}}
\newcommand{\PP}{\mathbb{P}}
\newcommand{\MMM}{\mathcal{M}}
\newcommand{\mm}{\mathfrak{m}}

\DeclareMathOperator{\Ker}{Ker}
\DeclareMathOperator{\GL}{GL}

\DeclareMathOperator{\Aut}{Aut}

\DeclareMathOperator{\Soc}{Soc}
\DeclareMathOperator{\Der}{Der}

\DeclareMathOperator{\Aff}{Aff}

\DeclareMathOperator{\Char}{char}

\theoremstyle{plain}

\newtheorem{proposition}{Proposition}
\newtheorem{theorem}{Theorem}
\newtheorem{corollary}{Corollary}
\theoremstyle{definition}

\theoremstyle{remark}
\newtheorem{remark}{Remark}

\usepackage{scalerel}
\newcommand{\ellipsesym}{\raisebox{1pt}{\vstretch{0.5}{\bigcirc}}}
\newcommand{\quartersym}{%
\mathord{%
    \vcenter{\hbox{%
        \begin{tikzpicture}[baseline=0pt, scale=0.25, rotate=45]
            \draw (0,0) -- (1,0) arc (0:90:1) -- cycle;
        \end{tikzpicture}%
}}}}
\newcommand{\hollowstar}{%
    \mathord{%
        \vcenter{\hbox{%
            \begin{tikzpicture}[baseline=0pt, scale=0.9]
                \node[star, star points=5, star point ratio=2.5, draw, 
                      inner sep=0pt, minimum size=2ex] {};
            \end{tikzpicture}%
        }}%
    }%
}

\usepackage{tikz}
\usetikzlibrary{calc}
\usetikzlibrary{shapes.geometric}
\usetikzlibrary{arrows.meta}
\usetikzlibrary{patterns}
\tikzstyle{nofill} = [fill=none,draw = black]
\tikzstyle{precornfill} = [fill=black!10, draw = black]
\tikzstyle{cornfill} = [fill=black!25, draw = black]

\tikzstyle{myyellow} = [circle, draw, color=yellow!75, fill=yellow!75, minimum size = 1.2*\d]
\newcommand{\myyellowform}{$\bigcirc$}
\tikzstyle{myblue} = [isosceles triangle, isosceles triangle apex angle=60, shape border rotate=90, yshift=-\d/4, draw, color=blue!30, fill=blue!30, minimum size = \d]
\newcommand{\myblueform}{$\triangle$}
\tikzstyle{mygreen} = [rectangle, draw, color=green!80!black!75!white, fill=green!80!black!75!white, minimum size = \d]
\newcommand{\mygreenform}{$\square$}
\tikzstyle{mypink} = [diamond, aspect=0.5, draw, color=red!40, fill=red!40, minimum width = 0.6*\d, minimum height = 1.2*\d]
\newcommand{\mypinkform}{$\Diamond$}
\tikzstyle{myorange} = [ellipse, draw, color=red!50!yellow!50!white, fill=red!50!yellow!50!white, minimum width = 1.2*\d, minimum height = 0.6*\d, inner sep=0pt]
\newcommand{\myorangeform}{$\ellipsesym$}
\tikzstyle{mycyan} = [star, star point ratio=2.5, draw, color=cyan!50!white, fill=cyan!50!white, minimum size = \d] 
\newcommand{\mycyanform}{$\hollowstar$}
\tikzstyle{myviolet} = [circular sector, circular sector angle=90, draw, color=magenta!50!blue!40!white, fill=magenta!50!blue!40!white, rotate=-90, minimum size = \d]
\newcommand{\myvioletform}{$\quartersym$}

\newcommand\drawcell[6]{%
    \filldraw[#4] ($2*#1*(e1)+2*#2*(e2)$) + (-\d,-\d) rectangle ++(\d,\d);
    \node[#5, scale=(#6+1)/2] (T) at ($2*#1*(e1)+2*#2*(e2)$) {}; 
    \draw ($2*#1*(e1)+2*#2*(e2)$) node{#3};
}

\newcommand\mymonom[2]{
    \ifnum#1>1
        \ifnum#2>1 x^#1y^#2
        \else \ifnum#2=0 x^#1 \else x^#1y^{\phantom{1}}\!\! \fi\fi
    \else \ifnum#1=0 
        \ifnum #2>1 y^#2
        \else \ifnum #2=0 1^{\phantom{1}}\!\! \else y^{\phantom{1}}\!\! \fi\fi
    \else
        \ifnum#2>1 xy^#2
        \else \ifnum#2=0 x^{\phantom{1}}\!\! \else xy^{\phantom{1}}\!\! \fi\fi
    \fi\fi
}

\newcounter{num}
\newcommand{\newstep}{\refstepcounter{num}\arabic{num}}

\begin{document}

\title[An infinite series of Gorenstein local algebras]{An infinite series of Gorenstein local algebras \\failing the affine homogeneity property}
\author{Roman Avdeev}
\address{%
{\bf Roman Avdeev} \newline\indent HSE University, Faculty of Computer Science, Pokrovsky Boulevard 11, Moscow, 109028 Russia}
\email{suselr@yandex.ru}
\author{Yulia Zaitseva}
\address{%
{\bf Yulia Zaitseva} \newline\indent HSE University, Faculty of Computer Science, Pokrovsky Boulevard 11, Moscow, 109028 Russia}
\email{yuliazaitseva@gmail.com}

\thanks{The study was implemented in the framework of the Basic Research Program at HSE University (HSE-BR-2025-84).}

%14L30: Group actions on varieties or schemes (quotients)
%13E10: Artinian rings and modules, finite-dimensional algebras
%13H10: Special types of local and semilocal rings (Cohen-Macaulay, Gorenstein, Buchsbaum, etc.)


\begin{abstract}
We provide an infinite series of commutative finite-dimensional Gorenstein local algebras $A_n$ for $n \ge 2$. We give an elementary proof that the maximal ideal of every algebra $A_n$ possesses a one-dimensional subspace that is different from the socle and invariant under the automorphism group of~$A_n$. The latter implies that the algebras~$A_n$ fail the affine homogeneity property. We also discuss some consequences concerning additive actions on projective hypersurfaces, related to the generalized Hassett–Tschinkel correspondence for these algebras.
\end{abstract}

\maketitle

%%%%%%%%%%%%%%%%%%%%%%%%%%%%%%%%%%%%%%
\section{Introduction}

We study commutative finite-dimensional local algebras over a field and their automorphism groups. Such algebras naturally arise in various mathematical contexts, including deformation theory, singularity theory, commutative algebra, and algebraic geometry. Their automorphism groups provide valuable insight into the symmetries of the underlying geometric or algebraic objects. 

Throughout this paper, we work over a field $\KK$ whose characteristic is denoted by~$\Char \KK$. All algebras are assumed to be associative, commutative, and with unity.

Let $A$ be a finite-dimensional local algebra over $\KK$ with maximal ideal~$\mm$ and residue field~$\KK$, so that $A = \KK \oplus \mm$ as vector spaces. By Nakayama's lemma for local rings, there is $N > 0$ such that $\mm^N = 0$.
In particular, all elements of $\mm$ are nilpotent. Let $\Aut(A)$ denote the automorphism group of~$A$. Recall that the \emph{socle} of~$A$ is the ideal $\Soc A = \{x \in A \mid x\mm = 0\}$.

The algebra $A$ is called \emph{Gorenstein} if $\dim \Soc A = 1$. The latter implies $\Soc A = \mm^s$ for the maximal integer~$s$ with $\mm^s \ne 0$. In what follows, we assume that every Gorenstein algebra $A$ satisfies $\mm \ne 0$, which is equivalent to $\Soc A \subseteq \mathfrak m$. In this case, every subspace $U \subseteq \mm$ satisfying $\mm = U \oplus \Soc A$ is called a \emph{complementary hyperplane}. 

For a Gorenstein algebra~$A$, certain geometric properties in the case~$\Char \KK = 0$ were studied by Fels and Kaup in~\cite{FK} and also by Isaev in~\cite{Isa}.  
Namely, given a complementary hyperplane $U \subseteq \mm$, one defines the \emph{nil-hypersurface} ${S_U = \ln(1 + U) \subseteq \mm}$, where $\ln$ denotes the standard logarithm series $\ln(1+\nobreak x) = \sum\limits_{k\ge1} (-1)^{k+1} \frac{1}{k}x^k$ and the sum is finite when applied to a nilpotent~$x$. Clearly, $S_U$ is a smooth affine hypersurface in~$\mm$. 
Consider the group of bijective affine transformations $\Aff(\mm) = \GL(\mm) \ltimes \mm$. A nil-hypersurface $S_U \subseteq \mm$ is called \emph{affinely homogeneous} if the  group $\Aff(S_U) = \{\phi\in \Aff(\mm) \mid \phi(S_U)=S_U\}$ acts on~$S_U$ transitively. According to~\cite[Corollary 4.10]{FK} or~\cite[Theorem~2.2]{Isa}, the property of $S_U$ being affinely homogeneous either holds or fails simultaneously for all complementary hyperplanes~$U \subseteq \mm$, and it is equivalent to the transitivity of the natural action of the group $\Aut(A)$ on the set of all such hyperplanes. In this case, $A$ is said to have the \emph{affine homogeneity property} or \emph{property~\textup{(AH)}} for short. Moreover, the latter condition enables one to extend the definition of property~(AH) to the case $\Char \KK > 0$.

Our motivation for studying property~(AH) comes from algebraic geometry. Suppose that $\KK$ is algebraically closed with $\Char \KK = 0$ and let $\GG_a = (\KK,+)$ be the additive group of~$\KK$. An action of the vector group $\GG_a^m$ on an algebraic variety~$X$ is called an \emph{additive action} if it is effective and admits a Zariski open orbit. Additive actions can be regarded as equivariant open embeddings of $\GG_a^m$ into algebraic varieties. The famous Hassett--Tschinkel correspondence establishes a remarkable bijection between additive actions on the projective space $\PP^m$ and local algebras of dimension~$m+1$; see~\cite[Section~2.4]{HT}, \cite[Proposition~5.1]{KL} for the original sources or~\cite[Section~1.6]{AZa} for a survey. We are interested in a generalization of this correspondence that describes induced additive actions on projective hypersurfaces in terms of H-pairs. To explain this correspondence, we need several notions.

An additive action $\GG_a^{m-1} \times X \to X$ on a projective hypersurface $X \subseteq \PP^m$ is called \emph{induced} if it can be extended to an action $\GG_a^{m-1} \times \PP^m \to \PP^m$. Two induced additive actions $\alpha_1\colon \GG_a^{m-1} \times X_1 \to X_1$ and $\alpha_2\colon \GG_a^{m-1} \times X_2 \to X_2$ are called \emph{equivalent} if there exist a group automorphism $\psi\colon \GG_a^{m-1} \to \GG_a^{m-1}$ and a variety automorphism $\phi\colon \PP^m \to \PP^m$ such that $\phi(X_1) = X_2$ and $\phi \circ \alpha_1 = \alpha_2 \circ (\psi\times\phi)$.
An \emph{H-pair} is a pair $(A,U)$, where $A$ is a finite-dimensional local algebra with maximal ideal~$\mm$ and $U \subseteq \mm$ is a hyperplane generating $A$ as an algebra. Two H-pairs $(A_1, U_1)$ and $(A_2, U_2)$ are called \emph{isomorphic} if there exists an algebra isomorphism $\phi\colon A_1 \to A_2$ such that $\phi(U_1)=U_2$.

Given an H-pair $(A,U)$, consider the natural projection $p \colon A \setminus \lbrace 0 \rbrace \to \PP(A)$ and put~$X = \overline{p(\exp U)}$, where the overline denotes the closure. Then $X$ is a hypersurface in~$\PP(A)$ and the natural action of $\exp U$ on~$\PP(A)$ by multiplication preserves~$X$ and induces an additive action on it. Now the generalized Hassett--Tschinkel correspondence asserts that this construction yields a bijection between equivalence classes of induced additive actions on hypersurfaces in~$\PP^m$ different from hyperplanes and isomorphism classes of H-pairs $(A,U)$ with $\dim A = m + 1$; see~\cite[Proposition~2.15]{HT} and also~\cite[Theorem~1]{AP}, \cite[Theorem 2.6]{AZa} for various formulations. 

It turns out that a considerable role in the study of induced additive actions on projective hypersurfaces is played by Gorenstein algebras.
We call a hypersurface $X \subseteq \PP^m$ \emph{nondegenerate} if there is no linear change of variables after which the equation of~$X$ involves fewer than $m + 1$ variables. For a hypersurface $X \subseteq \PP^m$ admitting an induced additive action and corresponding to an H-pair $(A,U)$, it is proved in~\cite[Theorem~2.30]{AZa} that $X$ is nondegenerate if and only if the algebra $A$ is Gorenstein and $U \subseteq \mm$ is a complementary hyperplane. Observe that in this case the two hypersurfaces $X \subseteq \PP^m$ and $S_U \subseteq \mm$ related to $(A,U)$ are connected to each other by the formula $X = \overline{p(\exp (\exp S_U - 1))}$. 

As can be seen, a Gorenstein algebra~$A$ has property~(AH) if and only if all H-pairs of the form $(A,U)$ are isomorphic to each other. 
In this case, the nondegenerate projective hypersurface~$X$ corresponding to an H-pair $(A,U)$ does not actually depend on~$U$ and hence is determined only by the algebra~$A$ itself.
Many examples of this kind can be found in~\cite[Section~3]{Bel}, \cite[Proposition~8]{ABZ}, and~\cite{Bel2}. Besides, it is known from~\cite[Corollary~4.11]{FK} and~\cite[Theorem 2.4]{Isa} that all Gorenstein algebras admitting a $\ZZ_{\ge 0}$-grading with one-dimensional component of degree~$0$ have property~(AH).
On the other hand, \cite[Theorem~2.32]{AZa} asserts that, up to equivalence, every nondegenerate projective hypersurface admits at most one induced additive action, so, for every Gorenstein algebra~$A$ failing property~(AH) and any two non-isomorphic H-pairs $(A,U_1)$, $(A,U_2)$, the two corresponding projective hypersurfaces are not projectively equivalent, that is, no automorphism of the ambient projective space maps one to the other.

Several examples of Gorenstein algebras failing property~(AH) are given in~\cite[Section~8]{FK}. 
For one of the examples, the Milnor algebra $\MMM(x^3 + x^2y^2 + y^4 + xz^2 + zu^2)$, a proof of failing property~(AH) is provided. The proof relies on computer computations.

In this paper, we exhibit an infinite series~$A_n$, $n\ge2$, of Gorenstein local algebras and provide an elementary proof of the fact that each of them fails property~(AH). Namely, for each $n \ge 2$, consider the 2-generated algebra 
\[A_n = \KK[x,y]/(y^{2n+3}, x^ny^2-y^{n+2}, x^{2n+1}-xy^{n+1}).\] 
It can be checked that for every $n\ge 2$ the algebra $A_n$ is finite-dimensional, local, and Gorenstein, its maximal ideal $\mm_n$ is generated by~$x,y$, and $\Soc A_n = \langle y^{2n+2}\rangle$; see Corollary~\ref{crl_fin-dim} and Proposition~\ref{An_descr_prop}. The construction is inspired by the algebra $\MMM(x^6 + x^2y^3 + y^5)$ mentioned in \cite[Section~8]{FK} for $\KK \in \lbrace \mathbb R, \mathbb C \rbrace$: in the Appendix we show that $\MMM(x^6 + x^2y^3 + y^5) \simeq \KK[x,y]/(y^7, 3x^2y^2+5y^4, 3x^5+xy^3)$, and for $\KK = \mathbb C$ the latter algebra is isomorphic to our algebra $A_2 = \KK[x,y]/(y^7, x^2y^2-y^4, x^5 - xy^3)$ via a suitable rescaling of the variables; see Proposition~\ref{M(F)} and Corollary~\ref{MMM_A2} for details.

We now present the main results of this paper.

\begin{theorem}
\label{groups_main_theor}
Suppose that $n\ge 2$ and $\Char \KK$ is either zero or coprime to both~$n$ and~$n-1$. Then the one-dimensional subspace $\langle y^{2n+1}\rangle \subseteq \mm_n$ is invariant with respect to the action of the automorphism group $\Aut(A_n)$. More precisely, every element $a \in \Aut(A_n)$ multiplies $y^{2n+1}$ by a scalar $\gamma_a$ satisfying $\gamma_a^{\frac{n-1}3} = 1$ for $n \equiv 1 \, (\mathrm{mod}\: 3)$ and $\gamma_a^{n-1} = 1$ otherwise.
\end{theorem}

The next two corollaries are immediate.

\begin{corollary} \label{U1U2_cor}
Under the assumptions of Theorem~\textup{\ref{groups_main_theor}}, let $U_1, U_2 \subseteq \mm_n$ be two complementary hyperplanes such that $y^{2n+1} \in U_1$ and $y^{2n+1} \notin U_2$. Then there exists no $\phi\in \Aut(A_n)$ such that $\phi(U_1)=U_2$.
\end{corollary}

\begin{corollary}
\label{notAH_cor}
Under the assumptions of Theorem~\textup{\ref{groups_main_theor}}, the algebra $A_n$ fails property~\textup{(AH)}.
\end{corollary}

In view of the above discussion of additive actions on projective hypersurfaces we also obtain

\begin{corollary}
\label{addact_cor}
Suppose that $n\ge 2$, $\KK$ is algebraically closed with $\Char\KK = 0$, and $U_1, U_2$ are as in Corollary~\textup{\ref{U1U2_cor}}. Then the H-pairs $(A_n, U_1)$ and $(A_n, U_2)$ correspond to additive actions on projective hypersurfaces $X_1, X_2 \subseteq \PP(A_n)$ that are not projectively equivalent.
\end{corollary}

We remark that, according to the construction in the generalized Hassett--Tschinkel correspondence, the additive actions on the hypersurfaces $X_1, X_2 \subseteq \PP(A_n)$ from Corollary~\ref{addact_cor} arise from the same additive action on $\PP(A_n)$ induced by the natural action of $\exp \mm_n$ on~$A_n$ by multiplication.
For each $X_i$, the additive action on it is obtained by passing to a suitable hyperplane subgroup in the acting group and then restricting the action to the subvariety.

As is well known, if $\KK$ is algebraically closed and $\Char \KK = 0$, then the Lie algebra of the group $\Aut(A_n)$ is given by the derivations of~$A_n$. In this case, it follows from Theorem~\ref{groups_main_theor} that the element $y^{2n+1} \in \mm_n$ is annihilated by all derivations of~$A_n$. In the next theorem, we obtain this result directly under the same restrictions on~$\KK$ as in Theorem~\ref{groups_main_theor}.

\begin{theorem}
\label{algebras_main_theor}
Suppose that $n\ge 2$ and $\Char \KK$ is either zero or coprime to both~$n$ and~$n-1$. Then the element $y^{2n+1} \in \mm_n$ is annihilated by all derivations of~$A_n$. 
\end{theorem}

We remark that the proof of Theorem~\ref{algebras_main_theor} is simpler than that of Theorem~\ref{groups_main_theor}, and, in the case of algebraically closed $\KK$ with $\Char \KK = 0$, Corollary~\ref{notAH_cor} can be deduced from Theorem~\ref{algebras_main_theor}. Indeed, in this case Theorem~\ref{algebras_main_theor} implies that the element $y^{2n+1} \in \mm_n$ is fixed by the connected component of the identity of the group~$\Aut(A_n)$, hence the orbit of $y^{2n+1}$ under the action of~$\Aut(A_n)$ is finite. It follows that $\Aut(A_n)$ cannot act transitively on the set of complementary hyperplanes in~$\mm_n$, so $A_n$ fails property~(AH).

This paper is organized as follows. In Section~\ref{sec_alg_desc}, we study the structure and establish main properties of the algebras~$A_n$. Sections~\ref{proof_groups_sec} and~\ref{proof_algebras_sec} are devoted to proofs of Theorems~\ref{groups_main_theor} and~\ref{algebras_main_theor}, respectively. We note that both proofs are elementary and do not involve computer computations. In Section~\ref{sec_example} we illustrate Corollary~\ref{addact_cor} for the algebra~$A_2$ by taking two particular complementary hyperplanes $U_1,U_2$ and writing down explicitly the equations defining the corresponding projective hypersurfaces~$X_1,X_2$ that are not projectively equivalent. Finally, in the Appendix we show that our algebra $A_2$ is isomorphic to the Milnor algebra $\MMM(x^6+x^2y^3+y^5)$ when $\KK$ is algebraically closed and $\Char \KK \notin \lbrace 2,3,5 \rbrace$.

\textbf{Acknowledgement.} The authors are grateful to Ivan Arzhantsev for suggesting the problem and useful discussions. Deep thanks are also due to the referee for valuable comments and suggestions on a previous version of this paper.

\section{The algebras~\texorpdfstring{$A_n$}{A\_n}}
\label{sec_alg_desc}

Fix the monomial basis $M = \lbrace x^iy^j \mid i,j \in \ZZ_{\ge0} \rbrace$ in $\KK[x,y]$.

For each $n \ge 2$, consider the ideal $I_n = (f_1,f_2,f_3) \triangleleft \KK[x,y]$ generated by the polynomials
\begin{equation} \label{eqn_f1f2f3}
f_1 = y^{2n+3}, \ f_2 = x^ny^2-y^{n+2}, \ f_3 = x^{2n+1}-xy^{n+1}.
\end{equation}
We note that $I_n$ is a binomial ideal in the sense of~\cite{ES}, that is, it is generated by binomials and monomials.
Let $A_n = \KK[x,y]/I_n$ be the corresponding quotient algebra. Consider also the polynomial $f_4 = xy^{n+3}$.

\begin{proposition} \label{prop_In}
The set $\lbrace f_1, f_2, f_3, f_4 \rbrace$ is a Gr\"obner basis of $I_n$ with respect to the lexicographical order $x \succ y$. In particular, $f_4 \in I_n$.
\end{proposition}

\begin{proof}
This is obtained by applying Buchberger's algorithm followed by elimination of superfluous generators; see, for example, \cite[Chapter~2, \S\,7, Theorem~2 and Lemma~3]{CLO1}.
\end{proof}

\begin{corollary} \label{crl_fin-dim}
For every $n\ge2$, the algebra $A_n$ is finite-dimensional. More precisely, $\dim A_n = n^2 + 6n+2$ and a basis of $A_n$ is given by the set of monomials
\begin{equation} \label{eqn_basis}
B_n = \lbrace g \in M \mid g \ \text{is not divisible by any of} \ y^{2n+3}, xy^{n+3}, x^ny^2, x^{2n+1} \rbrace.
\end{equation}
\end{corollary}

\begin{proof}
Note that $y^{2n+3}, xy^{n+3}, x^ny^2, x^{2n+1}$ are exactly the leading monomials of $f_1,f_4,f_2,f_3$, respectively. In view of Proposition~\ref{prop_In} this means that the monomials in~$B_n$ form a basis of~$A_n$. Observe that $B_n$ consists of all monomials $x^iy^j$ with $0 \le i \le 2n$, $0 \le j \le 1$, all monomials $x^iy^j$ with $0 \le i \le n-1$, $2 \le j \le n+2$, and all monomials $y^j$ with $n+3 \le j \le 2n+2$, so $A_n$ is finite-dimensional and $\dim A_n = (2n+1)\cdot2+n(n+1)+n = n^2+6n+2$.
\end{proof}

Next, we study relations in~$A_n$ between various monomials.
\begin{proposition} \label{prop_relations}
There are the following relations in~$A_n$:
\begin{align}
y^{n+2} &= x^ny^2; \label{green_eq} \\
y^{n+k+2} &= x^ny^{k+2} \ \text{for} \ 1 \le k \le n-2; \label{blue_eq}\\
y^{2n+1} &= x^ny^{n+1} = x^{3n}; \label{violet_eq}\\
y^{2n+2} &= x^ny^{n+2} = x^{2n}y^2 = x^{3n}y; \label{cyan_eq}\\
x^ky^{n+2} &= x^{n+k}y^2 = x^{2n+k}y \ \text{for} \ 1 \le k \le n-1; \label{pink_eq}\\
xy^{n+1} &= x^{2n+1}; \label{yellow_eq}\\
x^ky^{n+1} &= x^{2n+k} \ \text{for} \ 2 \le k \le n-1; \label{orange_eq}\\
y^{2n+3} &= xy^{n+3} = x^{n+1}y^3 = x^{2n+1}y^2 = x^{3n+1} = 0. \label{null_eq}
\end{align}
\end{proposition}

Before giving a proof of this proposition, we discuss a useful visualization of relations~(\ref{green_eq})--(\ref{null_eq}) shown in Figure~\ref{PictMain} below. Recall from~(\ref{eqn_basis}) the set $B_n$ of monomials forming a basis of~$A_n$ and consider the corresponding set 
\begin{equation} \label{eqn_Lambda_n}
\Lambda_n = \lbrace (i,j) \mid x^iy^j \in B_n \rbrace.
\end{equation}
In Figure~\ref{PictMain}, the set $\Lambda_n$ is given by the Young diagram bounded by the bold line. 

\begin{figure}[ht]
\begin{tikzpicture}[x=0.75pt,y=0.75pt,yscale=-1,xscale=1]
\fontsize{4pt}{5pt}\selectfont
\def\d{13}
\def\myn{4}
\coordinate (e1) at (\d,0);
\coordinate (e2) at (0,-\d);
\foreach \x/\y in {0/0,0/1,1/0,2/0,1/1,0/2} {
    \drawcell{\x}{\y}{$\mymonom{\x}{\y}$}{nofill}{}{0};
}
\foreach \x/\y in {3/0,2/1,1/2,0/3} {
    \draw ($2*\x*(e1)+2*\y*(e2)$) node{\small $\ddots$};
}
\drawcell{0}{\numexpr 2+\myn \relax}{$y^{n\!+\!2}$}{nofill}{mygreen}{1};
\drawcell{\myn}{2}{$x^{\!n}\!y^{\!2}$}{nofill}{mygreen}{1};
\drawcell{1}{\numexpr 1+\myn \relax}{$xy^{n\!+\!1}$}{nofill}{myyellow}{1};
\drawcell{\numexpr 1+2*\myn \relax}{0}{$x^{2n\!+\!1}$}{nofill}{myyellow}{1};
\ifnum \myn>3 \foreach \y in {\numexpr\myn+3\relax,...,\numexpr2*\myn\relax} {
    \pgfmathsetmacro{\coeff}{(\y-3-\myn)/(\myn-3)}
    \drawcell{0}{\y}{$\vdots$}{nofill}{myblue}{\coeff};
}\fi
\ifnum \myn>3 \foreach \y in {3,...,\myn} {
    \pgfmathsetmacro{\coeff}{(\y-3)/(\myn-3)}
    \drawcell{\myn}{\y}{$\vdots$}{nofill}{myblue}{\coeff};
}\fi
\pgfmathsetmacro{\coeff}{(\myn-2)/(\myn-1)}
\drawcell{0}{\numexpr2*\myn+1\relax}{$y^{2n\!+\!1}$}{nofill}{myviolet}{\coeff};
\drawcell{\myn}{\numexpr\myn+1\relax}{$x^{\!n}\!y^{\!n\!+\!1}$}{nofill}{myviolet}{\coeff};
\drawcell{\numexpr3*\myn\relax}{0}{$x^{3n}$}{nofill}{myviolet}{\coeff};
\drawcell{0}{\numexpr2*\myn+2\relax}{$y^{2n\!+\!2}$}{nofill}{mycyan}{1};
\drawcell{\myn}{\numexpr\myn+2\relax}{$x^{\!n}\!y^{\!n\!+\!2}$}{nofill}{mycyan}{1};
\drawcell{\numexpr2*\myn\relax}{2}{$x^{2n}y^2$}{nofill}{mycyan}{1};
\drawcell{\numexpr3*\myn\relax}{1}{$x^{3n}y$}{nofill}{mycyan}{1};
%Pink
\foreach \x in {1,...,\numexpr\myn-1\relax} {
    \ifnum \myn>2 \pgfmathsetmacro{\coeff}{(\x-1)/(\myn-2)} \else \pgfmathsetmacro{\coeff}{1} \fi
    \drawcell{\x}{\numexpr\myn+2\relax}{$\ldots$}{nofill}{mypink}{\coeff};
}
\foreach \x in {\numexpr\myn+1\relax,...,\numexpr2*\myn-1\relax} {
    \ifnum \myn>2 \pgfmathsetmacro{\coeff}{(\x-\myn-1)/(\myn-2)} \else \pgfmathsetmacro{\coeff}{1} \fi
    \drawcell{\x}{2}{$\ldots$}{nofill}{mypink}{\coeff};
}
\foreach \x in {\numexpr2*\myn+1\relax,...,\numexpr3*\myn-1\relax} {
    \ifnum \myn>2 \pgfmathsetmacro{\coeff}{(\x-2*\myn-1)/(\myn-2)} \else \pgfmathsetmacro{\coeff}{1} \fi
    \drawcell{\x}{1}{$\ldots$}{nofill}{mypink}{\coeff};
}
%Orange
\ifnum \myn>2
\foreach \x in {2,...,\numexpr\myn-1\relax} {
    \ifnum \myn>3 \pgfmathsetmacro{\coeff}{(\x-2)/(\myn-3)} \else \pgfmathsetmacro{\coeff}{1} \fi
    \drawcell{\x}{\numexpr\myn+1\relax}{$\ldots$}{nofill}{myorange}{\coeff};
}
\foreach \x in {\numexpr2*\myn+2\relax,...,\numexpr3*\myn-1\relax} {
    \ifnum \myn>3 \pgfmathsetmacro{\coeff}{(\x-2*\myn-2)/(\myn-3)} \else \pgfmathsetmacro{\coeff}{1} \fi
    \drawcell{\x}{0}{$\ldots$}{nofill}{myorange}{\coeff};
}\fi
%Zeros
\draw ($\numexpr(4*\myn+6)\relax*(e2)$) node{\small $0$};
\draw ($2*(e1)+\numexpr(2*\myn+6)\relax*(e2)$) node{\small $0$};
\draw ($\numexpr(2*\myn+2)\relax*(e1)+6*(e2)$) node{\small $0$};
\draw ($\numexpr(6*\myn+2)\relax*(e1)$) node{\small $0$};
\draw ($\numexpr(4*\myn+2)\relax*(e1)+4*(e2)$) node{\small $0$};
\draw[very thick] (-\d,\d) -- ++($\numexpr(2*2*\myn+2)\relax*(e1)$) -- ++($2*2*(e2)$) -- ++($\numexpr(-2*\myn-2)\relax*(e1)$) -- ++($\numexpr(2*\myn+2)\relax*(e2)$) -- ++($\numexpr(-2*\myn+2)\relax*(e1)$) -- ++($\numexpr(2*\myn)\relax*(e2)$) -- ++($-2*(e1)$) -- cycle;
\draw ($-2*(e1)+\numexpr(4*\myn+6)\relax*(e2)$) node {\scriptsize $\bf \Lambda_n$};
\end{tikzpicture}
\caption{The algebra $A_n$}
\label{PictMain}
\end{figure}

Let us give names for monomials appearing in~(\ref{green_eq})--(\ref{orange_eq}) as follows:
\begin{itemize}
\item \emph{\mygreenform-monomials} are $y^{n+2}, x^ny^2$ appearing in~(\ref{green_eq});
\item \emph{\myblueform-monomials} are $y^{n+k+2}, x^ny^{k+2}$ for $1 \le k \le n-2$ appearing in~(\ref{blue_eq});
\item \emph{\myvioletform-monomials} are $y^{2n+1}, x^ny^{n+1}, x^{3n}$ appearing in~(\ref{violet_eq});
\item \emph{\mycyanform-monomials} are $y^{2n+2}, x^ny^{n+2}, x^{2n}y^2, x^{3n}y$ appearing in~(\ref{cyan_eq});
\item \emph{\mypinkform-monomials} are $x^ky^{n+2}, x^{n+k}y^2, x^{2n+k}y$ for $1 \le k \le n-1$ appearing in~(\ref{pink_eq});
\item \emph{\myyellowform-monomials} are $xy^{n+1}, x^{2n+1}$ appearing in~(\ref{yellow_eq}); 
\item \emph{\myorangeform-monomials} are $x^ky^{n+1}, x^{2n+k}$ for $2 \le k \le n-1$ appearing in~(\ref{orange_eq}).
\end{itemize}
We remark that for $n=2$ there are no \myblueform-monomials and \myorangeform-monomials.
Now Figure~\ref{PictMain} visualizes relations~(\ref{green_eq})--(\ref{null_eq}) as follows:
\begin{itemize}
\item each monomial $x^iy^j$ corresponds to the cell with coordinates~$(i,j)$;
\item monomials in equally marked cells correspond to equal elements in~$A_n$ (relations~(\ref{green_eq})--(\ref{orange_eq}));
\item all monomials in the white zone outside the depicted figure are equal to~$0$ in~$A_n$ (as a consequence of relations~(\ref{null_eq})). 
\end{itemize}

\begin{proof}[Proof of Proposition~\textup{\ref{prop_relations}}]
We divide the proof into two steps. 

{\bf Step 1}. Note that the relation $f_2 \in I_n$ is equivalent to the equality $y^{n+2} = x^ny^2$ of \mygreenform-monomials in~$A_n$. Next, we multiply this equality by various monomials to obtain more relations in~$A_n$. Multiplying by powers of~$y$, we obtain the equalities $y^{n+k+2} = x^ny^{k+2}$ of \myblueform-monomials for $1 \le k \le n-2$, the equality $y^{2n+1} = x^ny^{n+1}$ of two \myvioletform-monomials, and the equality $y^{2n+2} = x^ny^{n+2}$ of two \mycyanform-monomials. Similarly, multiplying by powers of $x$ yields the equalities $x^ky^{n+2} = x^{n+k}y^2$ of \mypinkform-monomials for $1 \le k \le n-1$, the equality $x^ny^{n+2} = x^{2n}y^2$ of two \mycyanform-monomials, and also the equality $x^{n+1}y^{n+2} = x^{2n+1}y^2$. Multiplying by $xy$ and keeping in mind that $f_4 \in I_n$, we see that $xy^{n+3}=x^{n+1}y^3=0$. All monomials divisible by one of these two also equal~$0$; in particular, $x^{n+1}y^{n+1} = x^{n+1}y^{n+2} = 0$ and hence $x^{2n+1}y^2 = 0$. 
This argument is illustrated in Figure~\ref{PictProofOfMain}: the equality $y^{n+2} = x^ny^2$ of \mygreenform-monomials implies the equality of all monomials located in the green angles with vertices in these monomials.

\begin{figure}[h]
\begin{tikzpicture}[x=0.75pt,y=0.75pt,yscale=-1,xscale=1]
\fontsize{4pt}{5pt}\selectfont
\def\d{13}
\def\myn{4}
\coordinate (e1) at (\d,0);
\coordinate (e2) at (0,-\d);
\colorlet{darkgreen}{green!50!black!70!white}
\node[circular sector, circular sector angle=90, draw, color=darkgreen, fill=darkgreen, opacity=0.1, rotate=-135, minimum size = 9*\d] (T) at ($4*(e1)+16*(e2)$) {};
\node[circular sector, circular sector angle=90, draw, color=darkgreen, fill=darkgreen, opacity=0.1, rotate=-135, minimum size = 9*\d] (T) at ($12*(e1)+8*(e2)$) {};
\foreach \x/\y in {0/0,0/1,1/0,2/0,1/1,0/2} {
    \drawcell{\x}{\y}{$\mymonom{\x}{\y}$}{nofill}{}{0};
}
\foreach \x/\y in {3/0,2/1,1/2,0/3} {
    \draw ($2*\x*(e1)+2*\y*(e2)$) node{\small $\ddots$};
}
\drawcell{0}{\numexpr 2+\myn \relax}{$y^{n\!+\!2}$}{nofill}{mygreen}{1};
\drawcell{\myn}{2}{$x^{\!n}\!y^{\!2}$}{nofill}{mygreen}{1};
\ifnum \myn>3 \foreach \y in {\numexpr\myn+3\relax,...,\numexpr2*\myn\relax} {
    \pgfmathsetmacro{\coeff}{(\y-3-\myn)/(\myn-3)}
    \drawcell{0}{\y}{$\vdots$}{nofill}{myblue}{\coeff};
}\fi
\ifnum \myn>3 \foreach \y in {3,...,\myn} {
    \pgfmathsetmacro{\coeff}{(\y-3)/(\myn-3)}
    \drawcell{\myn}{\y}{$\vdots$}{nofill}{myblue}{\coeff};
}\fi
\pgfmathsetmacro{\coeff}{(\myn-2)/(\myn-1)}
\drawcell{0}{\numexpr2*\myn+1\relax}{$y^{2n\!+\!1}$}{nofill}{myviolet}{\coeff};
\drawcell{\myn}{\numexpr\myn+1\relax}{$x^{\!n}\!y^{\!n\!+\!1}$}{nofill}{myviolet}{\coeff};
\drawcell{0}{\numexpr2*\myn+2\relax}{$y^{2n\!+\!2}$}{nofill}{mycyan}{1};
\drawcell{\myn}{\numexpr\myn+2\relax}{$x^{\!n}\!y^{\!n\!+\!2}$}{nofill}{mycyan}{1};
\drawcell{\numexpr2*\myn\relax}{2}{$x^{2n}y^2$}{nofill}{mycyan}{1};
%Pink
\foreach \x in {1,...,\numexpr\myn-1\relax} {
    \ifnum \myn>2 \pgfmathsetmacro{\coeff}{(\x-1)/(\myn-2)} \else \pgfmathsetmacro{\coeff}{1} \fi
    \drawcell{\x}{\numexpr\myn+2\relax}{$\ldots$}{nofill}{mypink}{\coeff};
}
\foreach \x in {\numexpr\myn+1\relax,...,\numexpr2*\myn-1\relax} {
    \ifnum \myn>2 \pgfmathsetmacro{\coeff}{(\x-\myn-1)/(\myn-2)} \else \pgfmathsetmacro{\coeff}{1} \fi
    \drawcell{\x}{2}{$\ldots$}{nofill}{mypink}{\coeff};
}
\draw ($\numexpr(4*\myn+2)\relax*(e1)+4*(e2)$) node{\small $0$};
\draw (-\d,\d) -- ++($\numexpr(2*2*\myn+2)\relax*(e1)$) -- ++($2*2*(e2)$) -- ++($\numexpr(-2*\myn-2)\relax*(e1)$) -- ++($\numexpr(2*\myn+2)\relax*(e2)$) -- ++($\numexpr(-2*\myn+2)\relax*(e1)$) -- ++($\numexpr(2*\myn)\relax*(e2)$) -- ++($-2*(e1)$) -- cycle;
%Zeros
\draw ($\numexpr(4*\myn+6)\relax*(e2)$) node{\small $0$};
\draw ($2*(e1)+\numexpr(2*\myn+6)\relax*(e2)$) node{\small $0$};
\draw ($2*(e1)+\numexpr(2*\myn+8)\relax*(e2)$) node{\small $\vdots$};
\draw ($2*(e1)+\numexpr(4*\myn+4)\relax*(e2)$) node{\small $0$};
\draw ($2*(e1)+\numexpr(4*\myn+2)\relax*(e2)$) node{\small $0$};
\draw ($\numexpr(2*\myn+2)\relax*(e1)+6*(e2)$) node{\small $0$};
\draw ($\numexpr(2*\myn+2)\relax*(e1)+8*(e2)$) node{\small $\vdots$};
\draw ($\numexpr(2*\myn+2)\relax*(e1)+\numexpr(2*\myn+2)\relax*(e2)$) node{\small $0$};
\draw ($\numexpr(2*\myn+2)\relax*(e1)+\numexpr(2*\myn+4)\relax*(e2)$) node{\small $0$};
\draw[-{>[length=10pt,width=12pt]}, thick, darkgreen] ($7.4*(e1)+19.4*(e2)$) to[out=0, in=270] node[above right = 2pt]{\normalsize Step 1} ($17.4*(e1)+9.4*(e2)$);
\end{tikzpicture}
\caption{Proof of Proposition~\ref{prop_relations}, Step 1}
\label{PictProofOfMain}
\end{figure}

{\bf Step 2}. We now use the relation $f_3 \in I_n$, which is equivalent to the equality $xy^{n+1} = x^{2n+1}$ of \myyellowform-monomials. As above, we multiply this equality by various monomials and obtain more equalities in~$A_n$: $x^ky^{n+1} = x^{2n+k}$ of \myorangeform-monomials for $2 \le k \le n-1$, $x^ky^{n+2} = x^{2n+k}y$ of \mypinkform-monomials for $1 \le k \le n-1$, $x^ny^{n+1} = x^{3n}$ of \myvioletform-monomials, $x^ny^{n+2} = x^{3n}y$ of \mycyanform-monomials, and $x^{3n+1} = x^{n+1}y^{n+1} = 0$. This argument is illustrated in Figure~\ref{PictProofOfMain2}: the equality of \myyellowform{}-monomials $xy^{n+1} = x^{2n+1}$ implies equalities of all corresponding elements in angles with vertices in these monomials, which are depicted as yellow rectangles. 
\end{proof}

\begin{figure}[ht]
\begin{tikzpicture}[x=0.75pt,y=0.75pt,yscale=-1,xscale=1]
\fontsize{4pt}{5pt}\selectfont
\def\d{13}
\def\myn{4}
\coordinate (e1) at (\d,0);
\coordinate (e2) at (0,-\d);
\colorlet{darkgreen}{green!50!black!70!white}
\node[circular sector, circular sector angle=90, draw, color=darkgreen, fill=darkgreen, opacity=0.1, rotate=-135, minimum size = 9*\d] (T) at ($4*(e1)+16*(e2)$) {};
\node[circular sector, circular sector angle=90, draw, color=darkgreen, fill=darkgreen, opacity=0.1, rotate=-135, minimum size = 9*\d] (T) at ($12*(e1)+8*(e2)$) {};
\colorlet{darkyellow}{yellow!80!black}
\draw[very thick,rounded corners = 5pt, darkyellow, fill=darkyellow, opacity=0.2] ($(-\d,\d)+2*(e1)+\numexpr(2*\myn+2)\relax*(e2)$) rectangle ++($\numexpr(2*\myn+3)\relax*(e1)+\numexpr(2*\myn-1)\relax*(e2)$);
\draw[very thick,rounded corners = 5pt, darkyellow, fill=darkyellow, opacity=0.2] ($(-\d,\d)+\numexpr(4*\myn+2)\relax*(e1)$) rectangle ++($\numexpr(2*\myn+3)\relax*(e1)+\numexpr(2*\myn-1)\relax*(e2)$);

\foreach \x/\y in {0/0,0/1,1/0,2/0,1/1,0/2} {
    \drawcell{\x}{\y}{$\mymonom{\x}{\y}$}{nofill}{}{0};
}
\foreach \x/\y in {3/0,2/1,1/2,0/3} {
    \draw ($2*\x*(e1)+2*\y*(e2)$) node{\small $\ddots$};
}
\drawcell{0}{\numexpr 2+\myn \relax}{$y^{n\!+\!2}$}{nofill}{mygreen}{1};
\drawcell{\myn}{2}{$x^{\!n}\!y^{\!2}$}{nofill}{mygreen}{1};
\drawcell{1}{\numexpr 1+\myn \relax}{$xy^{n\!+\!1}$}{nofill}{myyellow}{1};
\drawcell{\numexpr 1+2*\myn \relax}{0}{$x^{2n\!+\!1}$}{nofill}{myyellow}{1};
\ifnum \myn>3 \foreach \y in {\numexpr\myn+3\relax,...,\numexpr2*\myn\relax} {
    \pgfmathsetmacro{\coeff}{(\y-3-\myn)/(\myn-3)}
    \drawcell{0}{\y}{$\vdots$}{nofill}{myblue}{\coeff};
}\fi
\ifnum \myn>3 \foreach \y in {3,...,\myn} {
    \pgfmathsetmacro{\coeff}{(\y-3)/(\myn-3)}
    \drawcell{\myn}{\y}{$\vdots$}{nofill}{myblue}{\coeff};
}\fi
\pgfmathsetmacro{\coeff}{(\myn-2)/(\myn-1)}
\drawcell{0}{\numexpr2*\myn+1\relax}{$y^{2n\!+\!1}$}{nofill}{myviolet}{\coeff};
\drawcell{\myn}{\numexpr\myn+1\relax}{$x^{\!n}\!y^{\!n\!+\!1}$}{nofill}{myviolet}{\coeff};
\drawcell{\numexpr3*\myn\relax}{0}{$x^{3n}$}{nofill}{myviolet}{\coeff};
\drawcell{0}{\numexpr2*\myn+2\relax}{$y^{2n\!+\!2}$}{nofill}{mycyan}{1};
\drawcell{\myn}{\numexpr\myn+2\relax}{$x^{\!n}\!y^{\!n\!+\!2}$}{nofill}{mycyan}{1};
\drawcell{\numexpr2*\myn\relax}{2}{$x^{2n}y^2$}{nofill}{mycyan}{1};
\drawcell{\numexpr3*\myn\relax}{1}{$x^{3n}y$}{nofill}{mycyan}{1};
%Pink
\foreach \x in {1,...,\numexpr\myn-1\relax} {
    \ifnum \myn>2 \pgfmathsetmacro{\coeff}{(\x-1)/(\myn-2)} \else \pgfmathsetmacro{\coeff}{1} \fi
    \drawcell{\x}{\numexpr\myn+2\relax}{$\ldots$}{nofill}{mypink}{\coeff};
}
\foreach \x in {\numexpr\myn+1\relax,...,\numexpr2*\myn-1\relax} {
    \ifnum \myn>2 \pgfmathsetmacro{\coeff}{(\x-\myn-1)/(\myn-2)} \else \pgfmathsetmacro{\coeff}{1} \fi
    \drawcell{\x}{2}{$\ldots$}{nofill}{mypink}{\coeff};
}
\foreach \x in {\numexpr2*\myn+1\relax,...,\numexpr3*\myn-1\relax} {
    \ifnum \myn>2 \pgfmathsetmacro{\coeff}{(\x-2*\myn-1)/(\myn-2)} \else \pgfmathsetmacro{\coeff}{1} \fi
    \drawcell{\x}{1}{$\ldots$}{nofill}{mypink}{\coeff};
}
%Orange
\ifnum \myn>2
\foreach \x in {2,...,\numexpr\myn-1\relax} {
    \ifnum \myn>3 \pgfmathsetmacro{\coeff}{(\x-2)/(\myn-3)} \else \pgfmathsetmacro{\coeff}{1} \fi
    \drawcell{\x}{\numexpr\myn+1\relax}{$\ldots$}{nofill}{myorange}{\coeff};
}
\foreach \x in {\numexpr2*\myn+2\relax,...,\numexpr3*\myn-1\relax} {
    \ifnum \myn>3 \pgfmathsetmacro{\coeff}{(\x-2*\myn-2)/(\myn-3)} \else \pgfmathsetmacro{\coeff}{1} \fi
    \drawcell{\x}{0}{$\ldots$}{nofill}{myorange}{\coeff};
}\fi
%Figure
\draw (-\d,\d) -- ++($\numexpr(2*2*\myn+2)\relax*(e1)$) -- ++($2*2*(e2)$) -- ++($\numexpr(-2*\myn-2)\relax*(e1)$) -- ++($\numexpr(2*\myn+2)\relax*(e2)$) -- ++($\numexpr(-2*\myn+2)\relax*(e1)$) -- ++($\numexpr(2*\myn)\relax*(e2)$) -- ++($-2*(e1)$) -- cycle;
%Zeros
\draw ($\numexpr(4*\myn+6)\relax*(e2)$) node{\small $0$};
\draw ($2*(e1)+\numexpr(2*\myn+6)\relax*(e2)$) node{\small $0$};
\draw ($2*(e1)+\numexpr(2*\myn+8)\relax*(e2)$) node{\small $\vdots$};
\draw ($2*(e1)+\numexpr(4*\myn+4)\relax*(e2)$) node{\small $0$};
\draw ($2*(e1)+\numexpr(4*\myn+2)\relax*(e2)$) node{\small $0$};
\draw ($\numexpr(2*\myn+2)\relax*(e1)+6*(e2)$) node{\small $0$};
\draw ($\numexpr(2*\myn+2)\relax*(e1)+8*(e2)$) node{\small $\vdots$};
\draw ($\numexpr(2*\myn+2)\relax*(e1)+\numexpr(2*\myn+2)\relax*(e2)$) node{\small $0$};
\draw ($\numexpr(2*\myn+2)\relax*(e1)+\numexpr(2*\myn+4)\relax*(e2)$) node{\small $0$};
\draw ($\numexpr(6*\myn+2)\relax*(e1)$) node{\small $0$};
\draw ($\numexpr(4*\myn+2)\relax*(e1)+4*(e2)$) node{\small $0$};
\draw ($\numexpr(6*\myn+2)\relax*(e1)+2*(e2)$) node{\small $0$};
\draw ($2*(e1)+\numexpr(4*\myn+2)\relax*(e2)$) node{\small $0$};
\draw[-{>[length=10pt,width=12pt]}, thick, darkgreen] ($7.4*(e1)+19.4*(e2)$) to[out=0, in=270] node[above right = 2pt]{\normalsize Step 1} ($17.4*(e1)+9.4*(e2)$);
\draw[-{>[length=10pt,width=12pt]}, thick, darkyellow] ($12.4*(e1)+15.4*(e2)$) to[out=0, in=270] node[above right = 2pt]{\normalsize Step 2} ($24.4*(e1)+6.4*(e2)$);
\end{tikzpicture}
\caption{Proof of Proposition~\ref{prop_relations}, Step 2}
\label{PictProofOfMain2}
\end{figure}

\begin{remark}
Every monomial appearing in relations \textup{(\ref{green_eq})--(\ref{orange_eq})} occurs in exactly one of them.
\end{remark}

\begin{corollary} \label{crl_3alt}
For every monomial $g \in M$, exactly one of the following three alternatives holds.
\begin{enumerate}[label=\textup{(\arabic*)},ref=\textup{\arabic*}]
\item 
$g = 0$ in~$A_n$. Moreover, this happens if and only if $g$ is divisible by one of the monomials $y^{2n+3}, xy^{n+3}, x^{n+1}y^3, x^{2n+1}y^2, x^{3n+1}$.

\item
$g \in B_n$.

\item
$g \notin B_n$ and there is a unique monomial $g_0 \in B_n$ such that $g = g_0$ in~$A_n$. Moreover, in this case $g$ and $g_0$ occur in the same relation among~\textup{(\ref{green_eq})--(\ref{orange_eq})}.
\end{enumerate}
\end{corollary}

\begin{remark}
It follows from Corollary~\ref{crl_3alt} that all linear relations in~$A_n$ between the monomials in~$M$ are spanned by relations~(\ref{green_eq})--(\ref{orange_eq}) and all relations of the form $g = 0$ with $g \in M$ divisible by one of $y^{2n+3}, xy^{n+3}, x^{n+1}y^3, x^{2n+1}y^2, x^{3n+1}$.
\end{remark}

For every monomial $g \in M$ we introduce the set
\begin{equation} \label{eqn_E(g)}
E(g) = \lbrace g' \in M \mid g = g' \ \text{in} \ A_n \rbrace.
\end{equation}

\begin{corollary} \label{crl_E(g)}
Suppose $g \in M$ is a monomial such that $g \ne 0$ in~$A_n$. Then the following assertions hold.
\begin{enumerate}[label=\textup{(\alph*)},ref=\textup{\alph*}]
\item
If $g$ occurs in a relation among~\textup{(\ref{green_eq})--(\ref{orange_eq})}, then $E(g)$ consists of all monomials appearing in that relation.
\item
In the remaining cases, $E(g) = \lbrace g \rbrace$.
\end{enumerate}
\end{corollary}

Recall from Corollary~\ref{crl_fin-dim} that $A_n$ is finite-dimensional. The next proposition describes the main properties of~$A_n$.

\begin{proposition}
\label{An_descr_prop}
For every $n \ge 2$, the algebra $A_n$ is local and Gorenstein. Moreover, the maximal ideal $\mm_n \triangleleft A_n$ is generated by $x,y$, the residue field of~$A_n$ is~$\KK$, and $\Soc A_n = \langle y^{2n+2}\rangle$.
\end{proposition}

\begin{proof}
Recall from~(\ref{null_eq}) that $x^{3n+1} = y^{2n+3} = 0$, so the ideal $\mm_n \triangleleft A_n$ generated by $x,y$ consists of nilpotent elements. As $A_n = \KK \oplus \mm_n$, it follows that the algebra $A_n$ is local, $\mm_n$ is its maximal ideal, and the residue field of~$A_n$ is~$\KK$.

\smallskip

It remains to prove that $\Soc A_n = \langle y^{2n+2}\rangle$. First we consider the linear map ${A_n \to A_n}$, $S \mapsto S \cdot x$, and analyze the images of the monomials in~$B_n$ using the alternatives of Corollary~\ref{crl_3alt}. The only elements of~$B_n$ mapped to~$0$ are the \myblueform-monomials $y^{n+k+2}$ for ${1 \le k \le n-2}$, the \myvioletform-monomial $y^{2n+1}$, and the \mycyanform-monomial $y^{2n+2}$. Next, the images of $y^{n+1}$ and $x^{2n}$ are equal to the same \myyellowform-monomial $xy^{n+1}$. Similarly, the images of $y^{n+2}$ and $x^{2n}y$ are equal to the same \mypinkform-monomial $xy^{n+2}$. All the remaining monomials are mapped to pairwise distinct monomials in~$B_n$ different from $xy^{n+1}$ and~$xy^{n+2}$. Now let $S = \sum\limits_{(i,j)\in\Lambda_n} s_{ij}x^iy^j \in \Soc A_n$ be an arbitrary element. It follows from the above analysis that the condition $S \cdot x = 0$ implies $s_{ij} = 0$ for all monomials $x^iy^j \in B_n$ not belonging to the set $\lbrace y^{n+k+2} \mid -1 \le k \le n \rbrace \cup \lbrace x^{2n}, x^{2n}y \rbrace$, which is shaded in Figure~\ref{PictProofOfMain3}. Moreover, $s_{0,n+1}+s_{2n,0} = s_{0,n+2}+s_{2n,1} = 0$, which yields
\[S = s_{0,n+1}(y^{n+1} - x^{2n}) + s_{0,n+2}(y^{n+2} - x^{2n}y) + \sum_{k=1}^n s_{0,n+k+2}y^{n+k+2}.\]

\begin{figure}[h]
\begin{tikzpicture}[x=0.75pt,y=0.75pt,yscale=-1,xscale=1]
\fontsize{4pt}{5pt}\selectfont
\def\d{13}
\def\myn{4}
\coordinate (e1) at (\d,0);
\coordinate (e2) at (0,-\d);
\foreach \x/\y in {0/0,0/1,1/0,2/0,1/1,0/2} {
    \drawcell{\x}{\y}{$\mymonom{\x}{\y}$}{nofill}{}{0};
}
\foreach \x/\y in {3/0,2/1,1/2,0/3} {
    \draw ($2*\x*(e1)+2*\y*(e2)$) node{\small $\ddots$};
}
\pgfmathsetmacro{\coeff}{(\myn-2)/(\myn-1)}
\drawcell{0}{\numexpr2*\myn+1\relax}{$y^{2n\!+\!1}$}{precornfill}{myviolet}{\coeff};
\drawcell{\myn}{\numexpr\myn+1\relax}{$x^{\!n}\!y^{\!n\!+\!1}$}{nofill}{myviolet}{\coeff};
\drawcell{\numexpr3*\myn\relax}{0}{$x^{3n}$}{nofill}{myviolet}{\coeff};
\drawcell{0}{\numexpr2*\myn+2\relax}{$y^{2n\!+\!2}$}{precornfill}{mycyan}{1};
\drawcell{\myn}{\numexpr\myn+2\relax}{$x^{\!n}\!y^{\!n\!+\!2}$}{nofill}{mycyan}{1};
\drawcell{\numexpr2*\myn\relax}{2}{$x^{2n}y^2$}{nofill}{mycyan}{1};
\drawcell{\numexpr3*\myn\relax}{1}{$x^{3n}y$}{nofill}{mycyan}{1};
\drawcell{0}{\numexpr 2+\myn \relax}{$y^{n\!+\!2}$}{precornfill}{mygreen}{1};
\drawcell{\myn}{2}{$x^{\!n}\!y^{\!2}$}{nofill}{mygreen}{1};
\drawcell{1}{\numexpr 1+\myn \relax}{$xy^{n\!+\!1}$}{nofill}{myyellow}{1};
\drawcell{\numexpr 1+2*\myn \relax}{0}{$x^{2n\!+\!1}$}{nofill}{myyellow}{1};
\ifnum \myn>3 \foreach \y in {\numexpr\myn+3\relax,...,\numexpr2*\myn\relax} {
    \pgfmathsetmacro{\coeff}{(\y-3-\myn)/(\myn-3)}
    \drawcell{0}{\y}{$\vdots$}{precornfill}{myblue}{\coeff};
}\fi
\ifnum \myn>3 \foreach \y in {3,...,\myn} {
    \pgfmathsetmacro{\coeff}{(\y-3)/(\myn-3)}
    \drawcell{\myn}{\y}{$\vdots$}{nofill}{myblue}{\coeff};
}\fi
%Pink
\foreach \x in {1,...,\numexpr\myn-1\relax} {
    \ifnum \myn>2 \pgfmathsetmacro{\coeff}{(\x-1)/(\myn-2)} \else \pgfmathsetmacro{\coeff}{1} \fi
    \drawcell{\x}{\numexpr\myn+2\relax}{$\ldots$}{nofill}{mypink}{\coeff};
}
\foreach \x in {\numexpr\myn+1\relax,...,\numexpr2*\myn-1\relax} {
    \ifnum \myn>2 \pgfmathsetmacro{\coeff}{(\x-\myn-1)/(\myn-2)} \else \pgfmathsetmacro{\coeff}{1} \fi
    \drawcell{\x}{2}{$\ldots$}{nofill}{mypink}{\coeff};
}
\foreach \x in {\numexpr2*\myn+1\relax,...,\numexpr3*\myn-1\relax} {
    \ifnum \myn>2 \pgfmathsetmacro{\coeff}{(\x-2*\myn-1)/(\myn-2)} \else \pgfmathsetmacro{\coeff}{1} \fi
    \drawcell{\x}{1}{$\ldots$}{nofill}{mypink}{\coeff};
}
%Orange
\ifnum \myn>2
\foreach \x in {2,...,\numexpr\myn-1\relax} {
    \ifnum \myn>3 \pgfmathsetmacro{\coeff}{(\x-2)/(\myn-3)} \else \pgfmathsetmacro{\coeff}{1} \fi
    \drawcell{\x}{\numexpr\myn+1\relax}{$\ldots$}{nofill}{myorange}{\coeff};
}
\foreach \x in {\numexpr2*\myn+2\relax,...,\numexpr3*\myn-1\relax} {
    \ifnum \myn>3 \pgfmathsetmacro{\coeff}{(\x-2*\myn-2)/(\myn-3)} \else \pgfmathsetmacro{\coeff}{1} \fi
    \drawcell{\x}{0}{$\ldots$}{nofill}{myorange}{\coeff};
}\fi
%Socle
\drawcell{0}{\numexpr \myn+1 \relax}{$y^{n+1}$}{precornfill}{}{0};
\drawcell{\numexpr 2*\myn \relax}{1}{$x^{2n}y$}{precornfill}{}{0};
\drawcell{\numexpr 2*\myn \relax}{0}{$x^{2n}$}{precornfill}{}{0};
%Zeros
\draw ($\numexpr(4*\myn+6)\relax*(e2)$) node{\small $0$};
\draw ($2*(e1)+\numexpr(2*\myn+6)\relax*(e2)$) node{\small $0$};
\draw ($\numexpr(2*\myn+2)\relax*(e1)+6*(e2)$) node{\small $0$};
\draw ($\numexpr(6*\myn+2)\relax*(e1)$) node{\small $0$};
\draw ($\numexpr(4*\myn+2)\relax*(e1)+4*(e2)$) node{\small $0$};
\draw[very thick] (-\d,\d) -- ++($\numexpr(2*2*\myn+2)\relax*(e1)$) -- ++($2*2*(e2)$) -- ++($\numexpr(-2*\myn-2)\relax*(e1)$) -- ++($\numexpr(2*\myn+2)\relax*(e2)$) -- ++($\numexpr(-2*\myn+2)\relax*(e1)$) -- ++($\numexpr(2*\myn)\relax*(e2)$) -- ++($-2*(e1)$) -- cycle;
\draw ($-2*(e1)+\numexpr(4*\myn+6)\relax*(e2)$) node {\scriptsize $\bf \Lambda_n$};
\end{tikzpicture}
\caption{Proof of Proposition~\ref{An_descr_prop}, step with $S \mapsto S \cdot x$}
\label{PictProofOfMain3}
\end{figure}

Next we use the condition $S\cdot y = 0$. Since
\begin{multline*}
0 = S\cdot y = s_{0,n+1}y^{n+2} - s_{0,n+1}x^{2n}y + s_{0,n+2}y^{n+3} - s_{0,n+2}x^{2n}y^2 + \sum_{k=1}^n s_{0,n+k+2}y^{n+k+3} =
\\
= - s_{0,n+1}x^{2n}y  + (s_{0,2n+1}- s_{0,n+2})y^{2n+2} +\sum_{k=-1}^{n-2} s_{0,n+k+2}y^{n+k+3},
\end{multline*}
we obtain $s_{0,2n+1}- s_{0,n+2} = 0$ and $s_{0,n+k+2} = 0$ for all $-1 \le k \le n-2$. In particular, for $k=0$ we get $s_{0,n+2}=0$, which yields $s_{0,2n+1}=0$. So $s_{0,n+k+2} = 0$ for all $-1 \le k \le n-1$ and we have $S = s_{0,2n+2}y^{2n+2}$. Thus $\Soc A_n = \langle y^{2n+2}\rangle$ and $A_n$ is Gorenstein. 
\end{proof}

\begin{remark} \label{rem_mm^}
As the ideal $\mm_n$ is generated by~$x,y$, it follows from Corollary~\ref{crl_3alt} that $\mm_n^{3n+2} = 0$ and $\mm_n^{3n+1}  = \langle x^{3n}y \rangle = \langle y^{2n+2} \rangle = \Soc A_n$.
\end{remark}

\section{Proof of Theorem~\ref{groups_main_theor}}
\label{proof_groups_sec}

Recall the notation and results of~Section~\ref{sec_alg_desc}.

Let $\phi\colon A_n \to A_n$ be an algebra automorphism and note that $\varphi$ is uniquely determined by the images of~$x,y$. Fix the expressions
$$\begin{aligned}
\phi(x) = \sum_{(i,j)\in \Lambda_n} a_{ij}x^iy^j \in A_n,\quad\quad
\phi(y) = \sum_{(i,j)\in \Lambda_n} b_{ij}x^iy^j \in A_n.
\end{aligned}$$
Since $x,y \in \mm_n \setminus \mm_n^2$ and $\varphi$ preserves all powers of~$\mm_n$, we have $\varphi(x), \varphi(y)  \in \mm_n \setminus \mm_n^2$. In particular, this implies that $a_{00} = b_{00} = 0$, at least one of $a_{10}, a_{01}$ is nonzero, and at least one of~$b_{10},b_{01}$ is nonzero.
We emphasize that $\phi(x), \phi(y)$ are regarded as elements of~$A_n$. In our arguments we will also work with the same polynomials regarded as elements of~$\KK[x,y]$, so we introduce separate notations for them:
\begin{equation}
\label{Phixy_eq}
\Phi_x = \sum_{(i,j)\in \Lambda_n \setminus\{(0,0)\}} a_{ij}x^iy^j \in \KK[x,y],\quad\quad
\Phi_y = \sum_{(i,j)\in \Lambda_n \setminus\{(0,0)\}} b_{ij}x^iy^j \in \KK[x,y].
\end{equation}
Consider the standard $\ZZ$-grading on $\KK[x,y]$ given by $\deg x^iy^j = i+j$ for all $i,j\ge 0$ and observe that all summands in $\Phi_x, \Phi_y$ have degree at least~$1$.

Our proof consists of 9 steps; see below.

Each of Steps~\ref{step1}--\ref{step8} follows the same general strategy: we take an element $f = \sum\limits_{i,j\ge0} c_{ij}x^iy^j \in I_n$ along with a monomial $g \in B_n$ and compute the coefficient $c$ of~$g$ in~$\varphi(f)$ regarded as an element of~$A_n$ and expressed as a linear combination of the monomials in~$B_n$.
Since $\varphi(f) = 0$ in~$A_n$, we have $c = 0$, which yields a relation on (some of) the coefficients $a_{ij}$ and~$b_{ij}$.
Recall from Corollary~\ref{crl_3alt} that the image in~$A_n$ of every monomial in~$M$ is either zero or equal to an element in~$B_n$. Using Corollary~\ref{crl_E(g)}, we determine the set~$E(g)$ defined by~(\ref{eqn_E(g)}). Then the computation of~$c$ reduces to the following problem: for each $g' \in E(g)$ and each pair $(i,j)$ with~$c_{ij} \ne 0$, determine the coefficient of~$g'$ in the expansion of the polynomial $\Phi_x^i\Phi_y^j$. In most cases, we have $\deg(g')= i+j$ or $\deg(g') = i+j+1$, hence the latter problem is solved by using the following two straightforward observations:
\begin{enumerate}[label=\textup{(R\arabic*)},ref=\textup{R\arabic*}]
\item \label{R1} the product of $k$ monomials of degree at least~$1$ has degree~$k$ if and only if all monomials in the product have degree~$1$;
\item \label{R2} the product of $k$ monomials of degree at least~$1$ has degree~$k+1$ if and only if one of the factors has degree~$2$ and the remaining factors have degree~$1$.
\end{enumerate}

Finally, at Step~\ref{step9} we complete the proof.

We now proceed to our arguments. Steps~\ref{step1}--\ref{step8} are performed according to the general strategy described above, so we apply it without further explanation.

{\bf Step~\newstep\label{step1}}. $f = f_3 = x^{2n+1} - xy^{n+1}$, $g = x^{n+2}$. Then $E(g) = \lbrace x^{n+2} \rbrace$. The coefficient of~$x^{n+2}$ in $\Phi_x^{2n+1}$ is obviously~$0$ and that in~$\Phi_x\Phi_y^{n+1}$ is $a_{10}b_{10}^{n+1}$ by~(\ref{R1}), hence $a_{10}b_{10} = 0$.

{\bf Step~\newstep\label{step2}}. $f = f_2 = x^ny^2 - y^{n+2}$, $g = x^{n+2}$. Then again $E(g) = \lbrace x^{n+2} \rbrace$. By~(\ref{R1}), the coefficient of~$x^{n+2}$ in $\Phi_x^{n}\Phi_y^2$ is $a_{10}^nb_{10}^2$ and that in~$\Phi_y^{n+2}$ is $b_{10}^{n+2}$, hence $a_{10}^nb_{10}^2 - b_{10}^{n+2} = 0$. As $a_{10}b_{10} = 0$, we get $b_{10} = 0$. Since $\phi(y) \in \mm_n \setminus \mm_n^2$, it also follows that $b_{01} \ne 0$.

{\bf Step~\newstep\label{step3}}. $f = f_3 = x^{2n+1} - xy^{n+1}$, $g = y^{n+2}$. Then $g$ is a \mygreenform-monomial, hence $E(g) = \lbrace y^{n+2}, x^ny^2 \rbrace$. Obviously, the coefficients of~$y^{n+2}$ and~$x^ny^2$ in $\Phi_x^{2n+1}$ are~$0$. Since $b_{10} = 0$, by~(\ref{R1}) the coefficients of~$y^{n+2}$ and $x^ny^2$ in~$\Phi_x\Phi_y^{n+1}$ are~$a_{01}b_{01}^{n+1}$ and~$0$, respectively. It follows that $a_{01}b_{01} = 0$. As $b_{01} \ne 0$, we obtain $a_{01} = 0$.  Since $\phi(x) \in \mm_n \setminus \mm_n^2$, it also follows that $a_{10} \ne 0$.

{\bf Step~\newstep\label{step4}}. $f = f_2 = x^ny^2 - y^{n+2}$, $g = y^{n+2}$. Then again $E(g) = \lbrace y^{n+2}, x^ny^2 \rbrace$. Keeping in mind that $a_{01} = b_{10} = 0$ and using~(\ref{R1}), we find that the coefficients of~$y^{n+2}$ and $x^ny^2$ in~$\Phi_x^n\Phi_y^2$ are $0$ and $a_{10}^nb_{01}^2$, respectively, and the coefficients of~$y^{n+2}$ and $x^ny^2$ in~$\Phi_y^{n+2}$ are $b_{01}^{n+2}$ and~$0$, respectively. It follows that $a_{10}^nb_{01}^2 = b_{01}^{n+2}$. As $b_{01} \ne 0$, we get $a_{10}^n = b_{01}^n$.

{\bf Step~\newstep\label{step5}}. $f = f_3 = x^{2n+1} - xy^{n+1}$, $g = xy^{n+1}$. Then $g$ is a \myyellowform-monomial, hence $E(g) = \lbrace xy^{n+1}, x^{2n+1} \rbrace$. As $a_{01} = 0$, all monomials in $\Phi_x^{2n+1}$ not divisible by~$x^2$ have degree at least $1 + 2\cdot2n > n+2$, hence the coefficient of~$xy^{n+1}$ in~$\Phi_x^{2n+1}$ is~$0$. By~(\ref{R1}), the coefficient of $x^{2n+1}$ in~$\Phi_x^{2n+1}$ is~$a_{10}^{2n+1}$ and the coefficient of~$xy^{n+1}$ in $\Phi_x\Phi_y^{n+1}$ is~$a_{10}b_{01}^{n+1}$. As $b_{10} = 0$, all monomials in~$\Phi_x\Phi_y^{n+1}$ not divisible by~$y$ have degree at least $1+2(n+1) > 2n+1$, hence the coefficient of~$x^{2n+1}$ in~$\Phi_x\Phi_y^{n+1}$ is~$0$. It follows that $a_{10}^{2n+1} = a_{10}b_{01}^{n+1}$. As $a_{10} \ne 0$, we obtain $a_{10}^{2n} = b_{01}^{n+1}$. Comparing this with the result of Step~\ref{step4}, we also find that $b_{01} = a_{10}^n$ and $a_{10}^{n(n-1)} = 1$.

{\bf Step~\newstep\label{step6}}. $f = f_2 = x^ny^2 - y^{n+2}$, $g = x^{n+2}y$. Then $E(g) = \lbrace x^{n+2}y \rbrace$. Keeping in mind that $a_{01} = b_{10} = 0$ and using~(\ref{R2}), we find that the coefficient of~$x^{n+2}y$ in~$\Phi_x^n\Phi_y^2$ is $2a_{10}^nb_{01}b_{20}$ and that in~$\Phi_y^{n+2}$ is~$0$. As $a_{10},b_{01} \ne 0$ and $2$ is coprime to~$\Char \KK$, it follows that $b_{20} = 0$.

{\bf Step~\newstep\label{step7}}. $f = f_2 = x^ny^2-y^{n+2}$, $g = y^{n+3}$. If $n \ge 3$, then $g$ is a \myblueform-monomial, hence $E(g) = \lbrace y^{n+3}, x^ny^3 \rbrace$.
If $n = 2$, then $g$ is a \myvioletform-monomial, hence $E(g) = \lbrace y^5, x^2y^3, x^6 \rbrace = \lbrace y^{n+3}, x^ny^3, x^6 \rbrace$. Keeping in mind that $a_{01} = b_{10} = 0$ and using~(\ref{R2}), we find that the coefficient of~$y^{n+3}$ in~$\Phi_x^n\Phi_y^2$ is~$0$ and that in~$\Phi_y^{n+2}$ is~$(n+2)b_{01}^{n+1}b_{02}$. Similarly, the coefficient of~$x^ny^3$ in~$\Phi_x^n\Phi_y^2$ is~$2a_{10}^nb_{01}b_{02} + na_{10}^{n-1}a_{11}b_{01}^2$ and that in~$\Phi_y^{n+2}$ is~$0$. (For $n=2$, the conclusion for~$\Phi_y^{n+2}$ also relies on $b_{20} = 0$.) Finally, since $b_{10} = b_{20} = 0$, in the case $n=2$ every monomial in~$\Phi_y$ not containing~$y$ is divisible by~$x^3$, hence the coefficient of~$x^6$ in both~$\Phi_x^2\Phi_y^2$ and $\Phi_y^4$ is~$0$. As a result, we obtain the relation $(n+2)b_{01}^{n+1}b_{02} = 2a_{10}^nb_{01}b_{02} + na_{10}^{n-1}a_{11}b_{01}^2$. Using the result of Step~\ref{step4}, we replace $b_{01}^{n+1}$ by~$a_{10}^nb_{01}$ in the left-hand side, which yields $na_{10}^nb_{01}b_{02} = na_{10}^{n-1}a_{11}b_{01}^2$. Since $a_{10},b_{01} \ne 0$ and $n$ is coprime to~$\Char \KK$, we can divide the latter equality by~$na_{10}^{n-1}b_{01} \ne 0$ and obtain $a_{10}b_{02} = a_{11}b_{01}$.

{\bf Step~\newstep\label{step8}}. $f = f_3 = x^{2n+1} - xy^{n+1}$, $g = xy^{n+2}$. Then $g$ is a \mypinkform-monomial, hence $E(g) = \lbrace xy^{n+2}, x^{n+1}y^2, x^{2n+1}y \rbrace$. As $a_{01} = 0$, every monomial in~$\Phi_x^{2n+1}$ divisible by~$y$ has degree at least $2n+2>n+3$, therefore the coefficients of~$xy^{n+2}$ and~$x^{n+1}y^2$ in $\Phi_x^{2n+1}$ are~$0$. Keeping in mind that $a_{01} = b_{10} = 0$ and using~(\ref{R2}), we find that the coefficient of~$x^{2n+1}y$ in~$\Phi_x^{2n+1}$ is~$(2n+1)a_{10}^{2n}a_{11}$. Similarly, the coefficients of~$xy^{n+2}$ and $x^{n+1}y^2$ in~$\Phi_x\Phi_y^{n+1}$ are $a_{11}b_{01}^{n+1} + (n+1)a_{10}b_{01}^nb_{02}$ and~$0$, respectively. (For $n=2$, the conclusion for $x^{n+1}y^2$ also relies on $b_{20} = 0$.) As $b_{10} = b_{20} = 0$, all monomials in~$\Phi_y^{n+1}$ not divisible by~$y^2$ have degree at least $1+3n$, therefore the coefficient of $x^{2n+1}y$ in~$\Phi_x\Phi_y^{n+1}$ is~$0$. It follows that $(2n+1)a_{10}^{2n}a_{11} = a_{11}b_{01}^{n+1} + (n+1)a_{10}b_{01}^nb_{02}$. In view of the result of Step~\ref{step5}, the left-hand side equals $(2n+1)a_{11}b_{01}^{n+1}$. After dividing by $b_{01}^n \ne 0$, the equality is reduced to the form $2na_{11}b_{01} = (n+1)a_{10}b_{02}$. Combining this with the result of Step~\ref{step7}, we get 
\[\left\{\begin{aligned}
a_{11}b_{01}&=b_{02}a_{10},\\
2na_{11}b_{01}&=(n+1)b_{02}a_{10}.
\end{aligned}\right.\]
Since $\det \left(\begin{smallmatrix}1 & 1 \\ 2n & n+1\end{smallmatrix}\right) = -n+1$ is coprime to $\Char \KK$, it follows that $a_{11}=b_{02}=0$.

{\bf Step~\newstep\label{step9}}. At this step, we finish the proof of the theorem by establishing all the required properties of~$\varphi(y^{2n+1})$. Recall that $y^{2n+1}$ is a \myvioletform-monomial, so we have $y^{2n+1} = x^{3n}$ in~$A_n$. We will show that $\varphi(x^{3n}) = a_{10}^{3n} x^{3n}$ in~$A_n$. Clearly, all monomials in the expansion of $\Phi_x^{3n}$ have degree at least~$3n$. On the other hand, the set of monomials in~$M$ that are nonzero in~$A_n$ and have degree at least $3n$ is $\lbrace x^{3n}, x^{3n-1}y, x^{3n}y \rbrace$ for $n \ge 3$ and $\lbrace x^6, x^5y, x^6y, x^4y^2, x^2y^4, y^6 \rbrace =  \lbrace x^{3n}, x^{3n-1}y, x^{3n}y, x^4y^2, x^2y^4, y^6 \rbrace$ for $n = 2$. As $a_{01} = 0$, using~(\ref{R1}) we find that the coefficients of $x^{3n}$ and $x^{3n-1}y$ in~$\Phi_x^{3n}$ are $a_{10}^{3n}$ and~$0$, respectively. As $a_{01} = a_{11} = 0$, using~(\ref{R2}) we find that the coefficient of~$x^{3n}y$ in~$\Phi_x^{3n}$ is~$0$, which proves the claim for~$n \ge 3$. If $n = 2$, in view of $a_{01} = 0$ and~(\ref{R1}) we easily see that the coefficients of the monomials $x^4y^2$, $x^2y^4$, $y^6$ in $\Phi_x^6$ are all zero, which yields the claim in this case as well. Recall from Step~\ref{step5} that $a_{10}^{n(n-1)} = 1$, so $(a_{10}^{3n})^{\frac{n-1}3} = 1$ for $n \equiv 1 \, (\mathrm{mod}\: 3)$ and $(a_{10}^{3n})^{n-1} = 1$ otherwise.

The proof of Theorem~\ref{groups_main_theor} is completed.

\begin{remark} \label{rem_aut}
The above proof provides some information on the structure of the (linear algebraic) group~$G_n = \Aut(A_n) \subseteq \mathrm{GL}(A_n)$.
Recall that for every $\varphi \in G_n$ we have $\varphi(x) \in a_{10}x+ \mm_n^2$, $\varphi(y) \in b_{01}y + \mm_n^2$ where $a_{10}^{n(n-1)} = 1$ and $b_{01} = a_{10}^n$.
Consider the finite multiplicative group
\[
C_n = \lbrace \lambda \in \KK \mid \lambda^{n(n-1)} = 1 \rbrace
\]
and observe that $C_n$ is identified with a subgroup of $G_n$ via the map $\lambda \mapsto \varphi_\lambda$ where $\varphi_\lambda(x) = \lambda x$, $\varphi_\lambda(y) = \lambda^ny$.
Next observe that the set
\[
N_n = \lbrace \varphi \in G_n \mid \varphi(x) \in x + \mm_n^2, \ \varphi(y) \in y + \mm_n^2 \rbrace
\]
is a closed subgroup of~$G_n$.
Clearly, $N_n$ (as well as the whole~$G_n$) preserves the flag of subspaces
\begin{equation} \label{flag_in_An}
0 = \mm_n^{3n+2} \subseteq \Soc A_n = \mm_n^{3n+1} \subseteq \mm_n^{3n} \subseteq \ldots \subseteq \mm_n^2 \subseteq \mm_n \subseteq A_n
\end{equation}
(the two equalities hold by Remark~\ref{rem_mm^}).
The definition of $N_n$ implies that it acts as the identity on each of the quotients $\mm_n^k / \mm_n^{k+1}$, $1 \le k \le 3n+1$, so the group $N_n$ is unipotent.
Now notice that $G_n$ decomposes into the semidirect product $G_n = C_n \ltimes N_n$.
In particular, the group $G_n$ is solvable. As $-1 \in C_n$, we conclude that $G_n$ is disconnected for all $n \ge 2$.
\end{remark}

\begin{remark} \label{rem_dim}
It would be interesting to compute or estimate the dimension of the group~$\Aut(A_n)$.
\end{remark}

\section{Proof of Theorem~\ref{algebras_main_theor}}
\label{proof_algebras_sec}

The proof resembles that of Theorem~\ref{groups_main_theor} (see Section~\ref{proof_groups_sec}), yet it is simpler.
Recall the notation and results of~Section~\ref{sec_alg_desc}.

Let $\partial \colon A_n \to A_n$ be a derivation. By the Leibniz rule, $\partial$ is uniquely determined by the images of~$x,y$. Fix the expressions
$$\begin{aligned}
\partial(x) = \sum_{(i,j)\in \Lambda_n} a_{ij}x^iy^j \in A_n,\quad\quad
\partial(y) = \sum_{(i,j)\in \Lambda_n} b_{ij}x^iy^j \in A_n.
\end{aligned}$$
We emphasize that $\partial(x), \partial(y)$ are regarded as elements of~$A_n$. In our arguments we will also work with the same polynomials regarded as elements of~$\KK[x,y]$, so we introduce separate notations for them:
\begin{equation}
\label{Dxy_eq}
D_x = \sum_{(i,j)\in \Lambda_n} a_{ij}x^iy^j \in \KK[x,y],\quad\quad
D_y = \sum_{(i,j)\in \Lambda_n} b_{ij}x^iy^j \in \KK[x,y].
\end{equation}
For every $f \in \KK[x,y]$, we introduce the polynomial $D_f = \frac{\partial f}{\partial x}D_x + \frac{\partial f}{\partial y} D_y$, so that $\partial (f) = D_f$ for $f$ and~$D_f$ regarded as elements of~$A_n$.

Our proof consists of 8 steps; see below.

Each of Steps~\ref{step1d}--\ref{step7d} follows the same general strategy: we take an element $f \in I_n$ along with a monomial $g \in B_n$ and compute the coefficient $c$ of~$g$ in~$\partial(f)$ regarded as an element of~$A_n$ and expressed as a linear combination of the monomials in~$B_n$.
Since $\partial(f) = 0$ in~$A_n$, we have $c = 0$, which yields a relation on (some of) the coefficients $a_{ij}$ and~$b_{ij}$.
Recall from Corollary~\ref{crl_3alt} that the image in~$A_n$ of every monomial in~$M$ is either zero or equal to an element in~$B_n$. Using Corollary~\ref{crl_E(g)}, we determine the set~$E(g)$ defined by~(\ref{eqn_E(g)}). Then the computation of~$c$ reduces to the following problem: for each $g' \in E(g)$, determine the coefficient of~$g'$ in the expansion of the polynomial~$D_f$. In fact, we have $f \in \lbrace f_2, f_3, f_4 \rbrace$ in all cases; for convenience, we write down the corresponding polynomials $D_f$ as follows:
\begin{align*}
D_{f_2} &= nx^{n-1}y^2D_x + 2x^nyD_y - (n+2)y^{n+1}D_y; \\
D_{f_3} &= (2n+1)x^{2n}D_x - y^{n+1}D_x - (n+1)xy^nD_y; \\
D_{f_4} &= y^{n+3}D_x + (n+3)xy^{n+2}D_y.
\end{align*}

Finally, at Step~\ref{step8d} we complete the proof.

We now proceed to our arguments. Steps~\ref{step1d}--\ref{step7d} are performed according to the general strategy described above, so we apply it without further explanation.

\setcounter{num}{0}

{\bf Step~\newstep\label{step1d}}. $f = f_4 = xy^{n+3}$, $g = y^{n+3}$.  If $n \ge 3$, then $g$ is a \myblueform-monomial, hence $E(g) = \lbrace y^{n+3}, x^ny^3 \rbrace$. If $n = 2$, then $g$ is a \myvioletform-monomial, hence $E(g) = \lbrace y^5, x^2y^3, x^6 \rbrace = \lbrace y^{n+3}, x^ny^3, x^6 \rbrace$. The coefficients of~$y^{n+3}$ and~$x^ny^3$ in~$D_f$ are $a_{00}$ and~$0$, respectively. If $n = 2$, then the coefficient of~$x^6$ in~$D_f$ is~$0$. We conclude that $a_{00} = 0$.

{\bf Step~\newstep\label{step2d}}. $f = f_2 = x^ny^2 - y^{n+2}$, $g = x^{n+1}y$. Then $E(g) = \lbrace x^{n+1}y \rbrace$. The coefficient of~$x^{n+1}y$ in~$D_f$ is~$2b_{10}$. As $2$ is coprime to~$\Char \KK$, we obtain $b_{10} = 0$.

{\bf Step~\newstep\label{step3d}}. $f = f_3 = x^{2n+1} - xy^{n+1}$, $g = y^{n+2}$. Then $g$ is a \mygreenform-monomial, hence $E(g) = \lbrace y^{n+2}, x^ny^2 \rbrace$. The coefficient of~$y^{n+2}$ in~$D_f$ is~$-a_{01}$. The coefficient of $x^ny^2$ in~$D_f$ is~$0$ for $n \ge 3$ and $-3b_{10}$ for $n = 2$. Hence we get $a_{01} = 0$ for $n \ge 3$ and $a_{01} = -3b_{10}$ for $n = 2$. Using the result of Step~\ref{step2d}, we obtain $a_{01} = 0$ for all $n \ge 2$.

{\bf Step~\newstep\label{step4d}}. $f = f_2 = x^ny^2 - y^{n+2}$, $g = y^{n+2}$. Then again $E(g) = \lbrace y^{n+2}, x^ny^2 \rbrace$. The coefficients of~$y^{n+2}$ and~$x^ny^2$ in~$D_f$ are $-(n+2)b_{01}$ and~$na_{10}+2b_{01}$, respectively. It follows that $na_{10} = nb_{01}$. Since $n$ is coprime to~$\Char \KK$, we obtain $a_{10} = b_{01}$.

{\bf Step~\newstep\label{step5d}}. $f = f_3 = x^{2n+1} - xy^{n+1}$, $g = xy^{n+1}$. Then $g$ is a \myyellowform-monomial, hence $E(g) = \lbrace xy^{n+1}, x^{2n+1} \rbrace$. The coefficients of~$xy^{n+1}$ and $x^{2n+1}$ in~$D_f$ are $-a_{10}-(n+1)b_{01}$ and $(2n+\nobreak1)a_{10}$, respectively. It follows that $2na_{10} = (n+1)b_{01}$. Combining this with the result of Step~\ref{step4d} we obtain $(n-1)a_{10} = 0$.  Since $n-1$ is coprime to~$\Char \KK$, it follows that $a_{10} = b_{01} = 0$.

{\bf Step~\newstep\label{step6d}}. $f = f_2 = x^ny^2 - y^{n+2}$, $g = y^{n+3}$. If $n \ge 3$, then $g$ is a \myblueform-monomial, hence $E(g) = \lbrace y^{n+3}, x^ny^3 \rbrace$. If $n = 2$, then $g$ is a \myvioletform-monomial, hence $E(g) = \lbrace y^5, x^2y^3, x^6 \rbrace = \lbrace y^{n+3}, x^ny^3, x^6 \rbrace$. The coefficients of~$y^{n+3}$ and~$x^ny^3$ in~$D_f$ are $-(n+2)b_{02}$ and $na_{11} + 2b_{02}$, respectively, and for $n = 2$ the coefficient of $x^6$ in~$D_f$ is~$0$. It follows that $na_{11} = nb_{02}$. Since $n$ is coprime to~$\Char \KK$, we obtain $a_{11} = b_{02}$.

{\bf Step~\newstep\label{step7d}}. $f = f_3 = x^{2n+1} - xy^{n+1}$, $g = xy^{n+2}$. Then $g$ is a \mypinkform-monomial, hence $E(g) = \lbrace xy^{n+2}, x^{n+1}y^2, x^{2n+1}y \rbrace$. The coefficients of~$xy^{n+2}$, $x^{n+1}y^2$, and $x^{2n+1}y$ in~$D_f$ are $-a_{11} - (n+1)b_{02}$, $0$, and~$(2n+1)a_{11}$, respectively. It follows that $2na_{11} = (n+1)b_{02}$. Combining this with the result of Step~\ref{step6d} we obtain $(n-1)a_{11} = 0$. Since $n-1$ is coprime to~$\Char \KK$, it follows that $a_{11} = b_{02} = 0$.

{\bf Step~\newstep\label{step8d}}. At this step, we finish the proof of the theorem by showing that $\partial(y^{2n+1}) = 0$. Recall that $y^{2n+1}$ is a \myvioletform-monomial, so we have $y^{2n+1} = x^{3n}$. Let us prove that $\partial(x^{3n}) = 0$. For $f = x^{3n}$ we have $D_f = 3nx^{3n-1}D_x$ in~$A_n$. We know from the previous steps that $a_{00} = a_{01} = a_{10} = a_{11} = 0$, therefore every monomial in~$D_x$ is divisible either by~$x^2$ or by~$y^2$. Then every monomial in $D_f$ is divisible either by~$x^{3n+1}$ or by~$x^{3n-1}y^2$, hence it equals~$0$ in~$A_n$ by relations~(\ref{null_eq}). Thus $\partial(x^{3n}) = 0$.

The proof of Theorem~\ref{algebras_main_theor} is completed.

\begin{remark}
Similarly to Remark~\ref{rem_aut}, the above proof provides some information about the structure of the Lie algebra $\Der(A_n) \subseteq \mathfrak{gl}(A_n)$.
For every $\partial \in \Der(A_n)$ we already know that $a_{00} = a_{10} = a_{01} = b_{10} = b_{01} = 0$.
Repeating the argument of Step~\ref{step2d} with $g = x^ny$, we obtain $b_{00} = 0$.
So we have $\partial(x), \partial(y) \in \mm_n^2$, which implies $\partial(\mm_n^k) \subset \mm_n^{k+1}$ for all $k \ge 1$.
In view of the flag of subspaces~(\ref{flag_in_An}) we conclude that the Lie algebra $\Der(A_n)$ is nilpotent.
\end{remark}

\begin{remark}
As in Remark~\ref{rem_dim}, it would be interesting to compute or estimate the dimension of the Lie algebra~$\Der(A_n)$.
\end{remark}

\section{Projectively non-equivalent hypersurfaces arising from the algebra~\texorpdfstring{$A_2$}{A\_2}}
\label{sec_example}

In this section, we assume that $\KK$ is algebraically closed and~$\Char \KK = 0$.

Recall from the Introduction that the generalized Hassett--Tschinkel correspondence establishes a one-to-one correspondence between isomorphism classes of H-pairs $(A,U)$ with $\dim A = m + 1$ and equivalence classes of induced additive actions on hypersurfaces $X \subseteq \PP^m$ different from hyperplanes. Moreover, $X$ is nondegenerate if and only if $A$ is Gorenstein and $U$ is a complementary hyperplane. If $A$ is Gorenstein and has property~(AH), then the nondegenerate hypersurface~$X$ corresponding to an H-pair $(A,U)$ does not actually depend on~$U$ and hence is determined only by~$A$ itself. In~\cite[Remark~2.31]{AZa}, the authors of that paper note that they do not have an example of non-equivalent additive actions on projective hypersurfaces $X_1, X_2$ that correspond to the same Gorenstein algebra~$A$. 
In this section, we provide an example of such hypersurfaces for the algebra~$A_2$ and write down their equations explicitly. Note that by~\cite[Theorem~2.32]{AZa} every nondegenerate projective hypersurface admits at most one induced additive action (up to equivalence), which implies that the hypersurfaces $X_1, X_2$ are not projectively equivalent. 

As explained in the Introduction, given a Gorenstein algebra~$A$ with maximal ideal~$\mm$ and a complementary hyperplane $U \subseteq \mm$, the projective hypersurface $X$ corresponding to the H-pair $(A,U)$ can be realized as the closure of $p(\exp U)$ in~$\PP(A)$, where $p \colon A \setminus \lbrace 0 \rbrace \to \PP(A)$ is the natural projection. According to~\cite[Theorem~5.1]{AS}, the degree of~$X$ equals the maximal integer $d$ such that $\mm^d \nsubseteq U$. As $A$ is Gorenstein, we have $\Soc A = \mm^s$ for the maximal integer $s$ with $\mm^s \ne 0$, whence $d = s$ and $d$ is uniquely determined by the property $\mm^d = \Soc A$.
Then, given $z_0 \in \KK$ and $z \in \mm$, the image in $\PP(A)$ of the element $z_0 + z \in A = \KK \oplus \mm$ belongs to~$X$ if and only if
\begin{equation} \label{eqn_hseqn}
z_0^d \pi\left(\ln\left(1+\tfrac{z}{z_0}\right)\right) = 0,
\end{equation}
where $\pi\colon \mm \to \mm/U \simeq \KK$ 
is the projection; see~\cite[Proposition~2.9]{AZa}. So, formula~(\ref{eqn_hseqn}) gives the equation of~$X$ in~$\PP(A)$.

We now proceed to our example. Consider the algebra~$A = A_2$ with $\dim A_2 = 18$ and maximal ideal~$\mm = \mm_2$. Put \(\Lambda_2' = \Lambda_2\setminus \{(0,0),(0,5),(0,6)\};\) see the shaded region in Figure~\ref{PictA2}. Recall from Proposition~\ref{An_descr_prop} that $\Soc A_2 = \langle y^6 \rangle$ and take two complementary hyperplanes $U_1, U_2 \subseteq \mm_2$ defined as follows:
\[\begin{aligned}
U_1 &= \KK y^5 \oplus \bigoplus_{(i,j) \in \Lambda_2'} \KK x^iy^j, \\
U_2 &= \KK(y^5-y^6) \oplus \bigoplus_{(i,j) \in \Lambda_2'} \KK x^iy^j.
\end{aligned}\]
Clearly, $y^5 \in U_1$ and $y^5 \notin U_2$, so $U_1,U_2$ satisfy the conditions of Corollary~\ref{addact_cor}.

\begin{figure}[ht]
\begin{tikzpicture}[x=0.75pt,y=0.75pt,yscale=-1,xscale=1]
\fontsize{4pt}{5pt}\selectfont
\def\d{13}
\def\myn{2}
\coordinate (e1) at (\d,0);
\coordinate (e2) at (0,-\d);
\drawcell{0}{0}{$\mymonom{0}{0}$}{nofill}{}{0};
\foreach \x/\y in {0/1,1/0,2/0,1/1,0/2,0/3,1/2,2/1,3/0,3/1,4/0,4/0,4/1} {
    \drawcell{\x}{\y}{$\mymonom{\x}{\y}$}{precornfill}{}{0};
}
\drawcell{0}{\numexpr 2+\myn \relax}{$y^4$}{precornfill}{mygreen}{1};
\drawcell{\myn}{2}{$x^2y^2$}{nofill}{mygreen}{1};
\drawcell{1}{\numexpr 1+\myn \relax}{$xy^3$}{precornfill}{myyellow}{1};
\drawcell{\numexpr 1+2*\myn \relax}{0}{$x^5$}{nofill}{myyellow}{1};
\ifnum \myn>3 \foreach \y in {\numexpr\myn+3\relax,...,\numexpr2*\myn\relax} {
    \pgfmathsetmacro{\coeff}{(\y-3-\myn)/(\myn-3)}
    \drawcell{0}{\y}{$\vdots$}{nofill}{myblue}{\coeff};
}\fi
\ifnum \myn>3 \foreach \y in {3,...,\myn} {
    \pgfmathsetmacro{\coeff}{(\y-3)/(\myn-3)}
    \drawcell{\myn}{\y}{$\vdots$}{nofill}{myblue}{\coeff};
}\fi
\pgfmathsetmacro{\coeff}{(\myn-2)/(\myn-1)}
\drawcell{0}{\numexpr2*\myn+1\relax}{$y^5$}{nofill}{myviolet}{\coeff};
\drawcell{\myn}{\numexpr\myn+1\relax}{$x^2y^3$}{nofill}{myviolet}{\coeff};
\drawcell{\numexpr3*\myn\relax}{0}{$x^6$}{nofill}{myviolet}{\coeff};
\drawcell{0}{\numexpr2*\myn+2\relax}{$y^6$}{nofill}{mycyan}{1};
\drawcell{\myn}{\numexpr\myn+2\relax}{$x^2y^4$}{nofill}{mycyan}{1};
\drawcell{\numexpr2*\myn\relax}{2}{$x^4y^2$}{nofill}{mycyan}{1};
\drawcell{\numexpr3*\myn\relax}{1}{$x^6y$}{nofill}{mycyan}{1};
%Pink
\foreach \x in {1,...,\numexpr\myn-1\relax} {
    \ifnum \myn>2 \pgfmathsetmacro{\coeff}{(\x-1)/(\myn-2)} \else \pgfmathsetmacro{\coeff}{1} \fi
    \drawcell{\x}{\numexpr\myn+2\relax}{$\mymonom{\x}{4}$}{precornfill}{mypink}{\coeff};
}
\foreach \x in {\numexpr\myn+1\relax,...,\numexpr2*\myn-1\relax} {
    \ifnum \myn>2 \pgfmathsetmacro{\coeff}{(\x-\myn-1)/(\myn-2)} \else \pgfmathsetmacro{\coeff}{1} \fi
    \drawcell{\x}{2}{$x^3y^2$}{nofill}{mypink}{\coeff};
}
\foreach \x in {\numexpr2*\myn+1\relax,...,\numexpr3*\myn-1\relax} {
    \ifnum \myn>2 \pgfmathsetmacro{\coeff}{(\x-2*\myn-1)/(\myn-2)} \else \pgfmathsetmacro{\coeff}{1} \fi
    \drawcell{\x}{1}{$x^5y$}{nofill}{mypink}{\coeff};
}
%Orange
\ifnum \myn>2
\foreach \x in {2,...,\numexpr\myn-1\relax} {
    \ifnum \myn>3 \pgfmathsetmacro{\coeff}{(\x-2)/(\myn-3)} \else \pgfmathsetmacro{\coeff}{1} \fi
    \drawcell{\x}{\numexpr\myn+1\relax}{$\ldots$}{nofill}{myorange}{\coeff};
}
\foreach \x in {\numexpr2*\myn+2\relax,...,\numexpr3*\myn-1\relax} {
    \ifnum \myn>3 \pgfmathsetmacro{\coeff}{(\x-2*\myn-2)/(\myn-3)} \else \pgfmathsetmacro{\coeff}{1} \fi
    \drawcell{\x}{0}{$\ldots$}{nofill}{myorange}{\coeff};
}\fi

\draw ($\numexpr(4*\myn+6)\relax*(e2)$) node{\small $0$};
\draw ($2*(e1)+\numexpr(2*\myn+6)\relax*(e2)$) node{\small $0$};
\draw ($\numexpr(2*\myn+2)\relax*(e1)+6*(e2)$) node{\small $0$};
\draw ($\numexpr(6*\myn+2)\relax*(e1)$) node{\small $0$};
\draw ($\numexpr(4*\myn+2)\relax*(e1)+4*(e2)$) node{\small $0$};
\draw[very thick] (-\d,\d) -- ++($\numexpr(2*2*\myn+2)\relax*(e1)$) -- ++($2*2*(e2)$) -- ++($\numexpr(-2*\myn-2)\relax*(e1)$) -- ++($\numexpr(2*\myn+2)\relax*(e2)$) -- ++($\numexpr(-2*\myn+2)\relax*(e1)$) -- ++($\numexpr(2*\myn)\relax*(e2)$) -- ++($-2*(e1)$) -- cycle;
\draw ($-2*(e1)+\numexpr(4*\myn+6)\relax*(e2)$) node {\scriptsize $\bf \Lambda_2$};
\draw[black!40!white] ($-2*(e1)+\numexpr(2*\myn+4)\relax*(e2)$) node {\scriptsize $\bf \Lambda_2'$};
\end{tikzpicture}
\caption{The algebra $A_2$}
\label{PictA2}
\end{figure}

For each $(i,j) \in \Lambda_2$, let $z_{ij}$ be the coordinate function on~$A_2$ corresponding to the basis element~$x^iy^j$. Then
$U_1=\{z\in \mm_2 \mid z_{06}=0\}$ and
$U_2=\{z\in \mm_2 \mid z_{05}+z_{06}=0\}$. Thus, for an element $z = \sum\limits_{(i,j)\in \Lambda_2 \setminus \lbrace (0,0) \rbrace} z_{ij}x^iy^j \in \mm_2$, the projections $\pi_i\colon \mm_2 \to \KK$, $i=1,2$, are given by $\pi_1(z)=z_{06}$ and $\pi_2(z)=z_{05}+z_{06}$.
By Remark~\ref{rem_mm^} we have $d=7$, so it remains to compute the coefficients $P_1$ and $P_2$ of $y^6$ and $y^5$, respectively, in the expression $z_0^7\ln(1+\tfrac{z}{z_0})=\sum\limits_{k=1}^7  \tfrac{(-1)^{k-1}}{k}z^kz_{00}^{7-k}$. Then the desired projectively non-equivalent hypersurfaces $X_1,X_2 \subseteq \PP^{17}$ are $X_1=\{P_1=0\}$ and $X_2= \{P_1+P_2=0\}$, where 
\scriptsize\begin{multline*}
P_1 = z_{01}z_{10}^6-
(\tfrac16z_{01}^6+\tfrac52z_{01}^4z_{10}^2+\tfrac52z_{01}^2z_{10}^4+5z_{01}z_{10}^4z_{20}+z_{10}^5z_{11})z_{00}+
\\
+(z_{01}^4z_{02}+z_{01}^4z_{20}+4z_{01}^3z_{10}z_{11}+6z_{01}^2z_{02}z_{10}^2+6z_{01}^2z_{10}^2z_{20}+ \\
4z_{01}z_{10}^3z_{11}+4z_{01}z_{10}^3z_{30}+6z_{01}z_{10}^2z_{20}^2+z_{02}z_{10}^4+z_{10}^4z_{21}+4z_{10}^3z_{11}z_{20})z_{00}^2-
\\
-(z_{01}^3z_{03}+z_{01}^3z_{21}+\tfrac32z_{01}^2z_{02}^2+3z_{01}^2z_{02}z_{20}+3z_{01}^2z_{10}z_{12}+3z_{01}^2z_{10}z_{30}+\tfrac32z_{01}^2z_{11}^2+\tfrac32z_{01}^2z_{20}^2+ \\
6z_{01}z_{02}z_{10}z_{11}+3z_{01}z_{03}z_{10}^2+3z_{01}z_{10}^2z_{21}+3z_{01}z_{10}^2z_{40}+6z_{01}z_{10}z_{11}z_{20}+6z_{01}z_{10}z_{20}z_{30}+z_{01}z_{20}^3+\\
\tfrac32z_{02}^2z_{10}^2+3z_{02}z_{10}^2z_{20}+z_{10}^3z_{12}+z_{10}^3z_{31}+\tfrac32z_{10}^2z_{11}^2+3z_{10}^2z_{11}z_{30}+3z_{10}^2z_{20}z_{21}+3z_{10}z_{11}z_{20}^2)z_{00}^3+
\\
+(z_{01}^2z_{04}+z_{01}^2z_{40}+2z_{01}z_{02}z_{03}+2z_{01}z_{02}z_{21}+2z_{01}z_{03}z_{20}+2z_{01}z_{10}z_{13}+2z_{01}z_{10}z_{31}+2z_{01}z_{11}z_{12}+\\
2z_{01}z_{11}z_{30}+2z_{01}z_{20}z_{21}+2z_{01}z_{20}z_{40}+z_{01}z_{30}^2+\tfrac13z_{02}^3+z_{02}^2z_{20}+2z_{02}z_{10}z_{12}+2z_{02}z_{10}z_{30}+z_{02}z_{11}^2+z_{02}z_{20}^2+2z_{03}z_{10}z_{11}+\\
z_{04}z_{10}^2+z_{10}^2z_{41}+2z_{10}z_{11}z_{21}+2z_{10}z_{11}z_{40}+2z_{10}z_{12}z_{20}+2z_{10}z_{20}z_{31}+2z_{10}z_{21}z_{30}+z_{11}^2z_{20}+2z_{11}z_{20}z_{30}+z_{20}^2z_{21})z_{00}^4-
\\
-(z_{02}z_{04}+z_{04}z_{20}+z_{01}z_{05}+z_{01}z_{41}+z_{02}z_{40}+\tfrac12z_{03}^2+z_{03}z_{21}+z_{10}z_{14}+z_{11}z_{13}+z_{11}z_{31}+\\
\tfrac12z_{12}^2+z_{12}z_{30}+z_{20}z_{41}+\tfrac12z_{21}^2+z_{21}z_{40}+z_{30}z_{31})z_{00}^5+
z_{06}z_{00}^6,
\end{multline*}\normalsize
\scriptsize\begin{multline*}
P_2=-\tfrac16z_{10}^6z_{00}+
(\tfrac15z_{01}^5+2z_{01}^3z_{10}^2+z_{10}^4z_{20})z_{00}^2-
(z_{01}^3z_{02}+z_{01}^3z_{20}+3z_{01}^2z_{10}z_{11}+3z_{01}z_{02}z_{10}^2+z_{10}^3z_{30}+\tfrac32z_{10}^2z_{20}^2)z_{00}^3+\\
+(z_{01}^2z_{03}+z_{01}^2z_{21}+z_{01}z_{02}^2+2z_{01}z_{02}z_{20}+2z_{01}z_{10}z_{12}+z_{01}z_{11}^2+2z_{02}z_{10}z_{11}+z_{03}z_{10}^2+z_{10}^2z_{40}+2z_{10}z_{20}z_{30}+\tfrac13z_{20}^3)z_{00}^4-\\
-(z_{01}z_{04}+z_{02}z_{03}+z_{02}z_{21}+z_{03}z_{20}+z_{10}z_{13}+z_{11}z_{12}+z_{20}z_{40}+\tfrac12z_{30}^2)z_{00}^5+
z_{05}z_{00}^6.
\end{multline*}\normalsize

The above expressions are obtained using the following Maple code:

\smallskip

\noindent
\verb|restart; |\\
\verb|with(Groebner): |\\
\verb|f1 := y^7: |\\
\verb|f2 := x^2*y^2-y^4: |\\
\verb|f3 := x^5-x*y^3: |\\
\verb|F := [f1, f2, f3]: |\\
\verb|GB := Basis(F, plex(x, y)): |\\
\verb|Z := y^6*z_06+y^5*z_05+y^4*z_04+y^3*z_03+y^2*z_02+y*z_01+x*y^4*z_14|\\
\verb|+x*y^3*z_13+x*y^2*z_12+x*y*z_11+x*z_10+x^2*y*z_21+x^2*z_20+x^3*y*z_31|\\
\verb|+x^3*z_30+x^4*y*z_41+x^4*z_40:|\\
\verb|Z_2 := simplify(NormalForm(Z*Z, GB, plex(x, y)), x, y): |\\
\verb|Z_3 := simplify(NormalForm(Z_2*Z, GB, plex(x, y)), x, y): |\\
\verb|Z_4 := simplify(NormalForm(Z_3*Z, GB, plex(x, y)), x, y): |\\
\verb|Z_5 := simplify(NormalForm(Z_4*Z, GB, plex(x, y)), x, y): |\\
\verb|Z_6 := simplify(NormalForm(Z_5*Z, GB, plex(x, y)), x, y): |\\
\verb|Z_7 := simplify(NormalForm(Z_6*Z, GB, plex(x, y)), x, y): |\\
\verb|LNz := 1/7*Z_7-1/6*Z_6*z_00+1/5*Z_5*z_00^2-1/4*Z_4*z_00^3+1/3*Z_3*z_00^4|\\
\hspace{2cm}\verb|-1/2*Z_2*z_00^5+Z*z_00^6: |\\
\verb|P_1 := coeff(LNz, y^6);|\\ 
\verb|P_2 := coeff(LNz, y^5);|

\appendix

\section{The origin of the algebra~\texorpdfstring{$A_2$}{A\_2}}

In this appendix, we explain why for $\KK = \mathbb C$ our algebra $A_2$ is isomorphic to the Milnor algebra $\MMM(x^6 + x^2y^3 + y^5)$ mentioned in \cite[Section~8]{FK}.
More precisely, the result holds whenever $\KK$ is algebraically closed and $\Char \KK \notin \lbrace 2,3,5 \rbrace$; see Corollary~\ref{MMM_A2} below.

For any field $\KK$ and a polynomial ring $R = \KK[x_1,\ldots,x_m]$, consider the maximal ideal $\mathfrak M = (x_1,\ldots,x_m) \triangleleft R$ and let $R_0$ be the localization of~$R$ at~$\mathfrak M$, so that
\[
R_0 = \lbrace f/g \mid f,g \in R, \ g(0,\ldots,0) \ne 0 \rbrace.
\]
Given a polynomial $F \in \mathfrak M^2$, consider its \emph{Jacobian ideal} $J(F) = (\frac{\partial F}{\partial x_1},\ldots, \frac{\partial F}{\partial x_m}) \triangleleft R_0$. Then the quotient $\MMM(F) = R_0/J(F)$ is called the \emph{Milnor algebra} of~$F$.

In what follows, we assume that $R = \KK[x,y]$ and $\Char \KK \notin \lbrace 2,3,5 \rbrace$. Consider the polynomial $F = x^6+x^2y^3+y^5 \in R$ and the corresponding Milnor algebra $\MMM(F) = \MMM(x^6+x^2y^3+y^5)$. Put $F_1 = \frac12\frac{\partial F}{\partial x} = 3x^5+xy^3$ and $F_2 = \frac{\partial F}{\partial y} = 3x^2y^2+5y^4$, so that $J(F) = (F_1,F_2) \triangleleft R_0$.
In the next proposition, we express $\MMM(F) = R_0/J(F)$ as a quotient of the ring~$R$.

\begin{proposition} \label{M(F)}
Suppose that $\Char \KK \notin \lbrace 2,3,5 \rbrace$. Then $\MMM(F) \simeq R / (F_1,F_2,y^7)$.
\end{proposition}

\begin{proof}
Fix the antigraded reverse lexicographical order on the set of monomials~$M = \lbrace x^iy^j \mid i,j \in \ZZ_{\ge0}\rbrace$ (see~\cite[Chapter~4, Definition~3.3]{CLO2}) and recall the notion of a standard basis of an ideal of the ring~$R_0$ (see~\cite[Chapter~4, Definition~4.1]{CLO2}).
Using the analog of Buchberger's algorithm~\cite[Chapter~4, Theorem~4.2(b)]{CLO2}, which also involves the Mora normal form algorithm~\cite[Chapter~4, Corollary~3.13]{CLO2}, we compute a polynomial standard basis $S$ of the ideal~$J(F)\triangleleft R_0$: $S = \lbrace F_1,F_2,F_3,F_4 \rbrace$ where $F_3 = 5y^5-9x^6$ and $F_4 = 3x^7+5x^5y^2$.
Note that the leading monomials of $F_1,F_2,F_3,F_4$ are $xy^3,3x^2y^2,5y^5,3x^7$, respectively.
Applying \cite[Chapter~4, Theorem~4.3]{CLO2} we find that a basis of $\MMM(F)$ is given by all monomials in the set $\lbrace g \in M \mid g \ \text{is not divisible by any of} \ xy^3,x^2y^2,y^5,x^7 \rbrace$, hence $\dim \MMM(F) = 18$.
Next we use the Mora normal form algorithm to reduce $x^7$ with respect to~$S$, which yields $x^7 \in J(F)$ and
\[
x^7 = \frac59\cdot\frac{-25y^3 F_1+(-3x^3+5xy^2)F_2}{3+25y}+\frac13 F_4.
\]
The standard basis~$S$, the fact that $x^7 \in J(F)$, and the above expression for $x^7$ can also be obtained using the computer algebra system Singular~\cite{Sin}; the corresponding code (for $\Char \KK = 0$) is as follows:

\smallskip

\noindent
\verb|ring r = 0, (x,y), ds;|\\
\verb|poly f = x6+x2y3+y5;|\\
\verb|ideal i = std(jacob(f));|\\
\verb|i;|\\
\verb|division(x7,i);|

\smallskip

\noindent
It follows that another standard basis of $J(F)$ is given by the set $\widetilde S = \lbrace F_1,F_2,F_3,\widetilde F_4 \rbrace$ where $\widetilde F_4 = x^7$.

Now consider the ideal $I = (\widetilde S) \triangleleft R$ and let $\varphi \colon R \to \MMM(F)$ be the ring homomorphism defined by the chain of natural maps $R \to R_0 \to \MMM(F)$.
Clearly, $\varphi$ is surjective and $I \subset \Ker \varphi$, so we get a well-defined surjective homomorphism $\psi \colon R/I \to \MMM(F)$.
Applying~\cite[Chapter~2, \S\,7, Theorem~2 and Lemma~3]{CLO1} we find that a Gr\"obner basis of $I$ with respect to the lexicographical order $x \succ y$ is $\lbrace F_1,F_2,xy^5,y^7 \rbrace$.
Then a basis of $R/I$ is given by all monomials in the set $\lbrace g \in M \mid g \ \text{is not divisible by any of} \ x^5,x^2y^2,xy^5,y^7 \rbrace$, hence $\dim R/I = 18$.
Since $\dim \MMM(F) = \dim R/I$, it follows that $\psi$ is an isomorphism.
Finally, we check that $\lbrace F_1,F_2,xy^5,y^7 \rbrace$ is also a Gr\"obner basis of the ideal $(F_1,F_2,y^7) \triangleleft R$, which yields $I = (F_1,F_2,y^7)$.
\end{proof}

\begin{corollary} \label{MMM_A2}
Suppose $\Char \KK \notin \lbrace 2,3,5 \rbrace$ and $\KK$ is algebraically closed. Then $\MMM(F) \simeq A_2$.
\end{corollary}

\begin{proof}
The automorphism of $\KK[x,y]$ given by $x \mapsto \sqrt{-\frac3{125}}x, y \mapsto -\frac3{25}y$ takes the ideal $I$ to $(x^5-xy^3, x^2y^2-y^4,y^7)$, which yields $\MMM(F) \simeq A_2$.
\end{proof}

%%%%%%%%%%%%%%%%%%%%%%%%%%%%%%%%%%%%%%%%

\end{document}